\documentclass[aop, preprint, 12pt]{imsart}

\RequirePackage[OT1]{fontenc}
\RequirePackage{amsthm,amsmath}
\RequirePackage[numbers]{natbib}
\RequirePackage[colorlinks,citecolor=blue,urlcolor=blue]{hyperref}
\usepackage{amsfonts}
\usepackage{color}
\usepackage{pifont}
\usepackage{amssymb}
\usepackage[english]{babel}
\usepackage[T1]{fontenc}
\usepackage{graphicx}
\usepackage{subfigure}
\usepackage{latexsym}
\usepackage{amstext}
\usepackage{mathrsfs}
\usepackage{hyperref}
\usepackage{dsfont}
\usepackage{graphics}
\usepackage{enumerate}
\usepackage{paralist}
\usepackage{color}
\usepackage{fullpage}
\newtheorem{prop}{Proposition}[section]

\numberwithin{equation}{section}

\newcommand{\beq}{\begin{eqnarray}}
\newcommand{\beqq}{\begin{eqnarray*}}
\newcommand{\eeq}{\end{eqnarray}}
\newcommand{\eeqq}{\end{eqnarray*}}

\usepackage{mathtools}
\usepackage[english]{babel}
\usepackage[utf8]{inputenc}
\usepackage{amsmath}
\usepackage{graphicx}
\usepackage[colorinlistoftodos]{todonotes}
\usepackage{polynom}
\usepackage{amsfonts}
\usepackage{dsfont}
\usepackage{amsthm}
\usepackage{hyperref}
\usepackage{graphicx}
\newtheorem{theorem}{Theorem}[section]

\newtheorem{lemma}{Lemma}[section]

\newtheorem{assumptions}[theorem]{Assumption}

\usepackage{fullpage}
\usepackage[T1]{fontenc}
\usepackage{hyperref}
\usepackage{xcolor}
\definecolor{link-color}{rgb}{0.15,0.4,0.15}
\hypersetup{
  colorlinks,
  linkcolor = {link-color}, citecolor = {link-color},
}

\usepackage{graphicx}
\usepackage{hyperref}
\usepackage{amsmath,amsthm,amssymb}
\usepackage{bbm}
\usepackage{tabularx}

\newcommand{\N}{\mathbb{N}}

\newcommand{\E}{\mathbb{E}}

\DeclareMathOperator{\er}{\mathbf{E}}

    \def\d{{\textnormal d}}

\newenvironment{eqnarr}{\begin{IEEEeqnarray}{rCl}}{\end{IEEEeqnarray}\ignorespacesafterend}

\renewcommand{\eqref}[1]{\hyperref[#1]{(\ref*{#1})}}

\newcommand*{\norm}[1]{\lVert #1 \rVert}

    \def\beq{\begin{eqnarr}}
    \def\eeq{\end{eqnarr}}
    \def\beqq{\begin{eqnarray*}} 
    \def\eeqq{\end{eqnarray*}} 

        \def\d{{\rm d}}

    \def\d{{\textnormal d}}
\newtheorem{remark}{Remark}[section]

\newcommand*{\pref}[1]{\hyperref[#1]{(\ref*{#1})}}
\newcommand*{\refpref}[2]{\hyperref[#2]{\ref*{#1}(\ref*{#2})}}

  \newcommand{\D}{{\rm d}}

\newcommand{\fT}{\overset{_\rightarrow}{\texttt T}}

\newcommand{\fS}{\overset{_\rightarrow}{\texttt S}}

\newcommand{\fF}{\overset{_\rightarrow}{\texttt F}}

\newcommand{\bT}{\overset{_\leftarrow}{\texttt T}}

\newcommand{\T}{\texttt T}

\newcommand{\M}{\texttt M}

\newcommand{\K}{\texttt K}

\newcommand{\U}{\texttt U}

\newcommand{\bS}{\overset{_\leftarrow}{\texttt S}}
\newcommand{\ebS}{\overset{_\leftarrow}{\emph{\texttt S}}}

\newcommand{\bF}{\overset{_\leftarrow}{\texttt F}}
\newcommand{\ebF}{\overset{_\leftarrow}{\emph{\texttt F}}}

\newcommand{\bL}{\overset{_\leftarrow}{\texttt L}}

\newcommand{\diagonal}{\texttt{diag}}

\newcommand{\bA}{\overset{_\leftarrow}{\texttt A}}
\newcommand{\fA}{\overset{_\rightarrow}{\texttt A}}

\newcommand{\ebA}{\overset{_\leftarrow}{\emph{\texttt A}}}

\startlocaldefs
\numberwithin{equation}{section}
\theoremstyle{plain}

\endlocaldefs

\begin{document}

\begin{frontmatter}
\title{Multi-species neutron transport equation}

\runtitle{Recurrent extension for ssMp in a Wedge}

\begin{aug}
\author{\fnms{Alexander M.G. Cox}\thanksref{t3}\ead[label=e1]{a.m.g.cox@bath.ac.uk}}, \author{\fnms{Simon C. Harris}\thanksref{t3}\ead[label=e2]{s.c.harris@bath.ac.uk}},\\
\author{\fnms{Emma L. Horton}\ead[label=e5]{e.l.horton@bath.ac.uk}}
\and \author{\fnms{Andreas E. Kyprianou}\thanksref{t3}\ead[label=e3]{a.kyprianou@bath.ac.uk}}

\thankstext{t3}{Supported by EPSRC grant EP/P009220/1. Emma Horton is supported by a PhD scholarship with joint funding from {\it EPSRC CDT SAMBa} and industrial partner {\it Wood} (formerly {\it Amec Foster Wheeler}).}



\address{
A. M. G. Cox,  \\
E. L. Horton,\\
A.E. Kyprianou and \\
University of Bath\\
Department of Mathematical Sciences \\
Bath, BA2 7AY\\
 UK.\\
\printead{e1}\\
\printead{e5}\\
\printead{e3}\\
}

\address{S. C. Harris\\
Department of Statistics\\
University of Auckland\\
Private Bag 92019\\
Auckland 1142\\
New Zealand\\
\printead{e2}
}
\end{aug}

\begin{abstract}\hspace{0.1cm}
The Neutron Transport Equation (NTE) describes the flux of neutrons through inhomogeneous fissile medium. Whilst well treated in the nuclear physics literature (cf. \cite{DS, PP}), the NTE has had a somewhat scattered treatment in mathematical literature with a variety of different approaches (cf. \cite{DL6, M-K}). Within a probabilistic framework it has somewhat undeservingly received little attention in recent years; nonetheless,   probabilistic treatments can be found see for example \cite{LPS, MWY, MT, PR, BLP, BLP2}. In this article our aim is threefold. First we want to introduce a slightly more general setting for the NTE, which gives a more complete picture of the different species of particle and radioactive fluxes that are involved in fission. Second we consolidate the classical $c_0$-semigroup approach to solving the NTE with the  method of stochastic representation which involves expectation semigroups. Third we provide the leading asymptotic of our multi-species NTE, which will turn out to be crucial for further stochastic analysis of the NTE in forthcoming work \cite{SNTE, SNTE2, MCNTE}. The methodology used in this paper harmonises the culture of expectation semigroup analysis from the theory of stochastic processes against $c_0$-semigroup theory from functional analysis. In this respect, our presentation is thus part review of existing theory and part presentation of new research results based on generalisation of existing results.
\end{abstract}



\end{frontmatter}

\section{Introduction}
The neutron transport equation (NTE) describes the flux of neutrons across a directional planar cross-section  in an inhomogeneous fissile medium (typically measured is number of neutrons per cm$^2$ per second). 
As such, flux is described as a function of time, $t$, Euclidian location, $r\in \mathbb{R}^3$, direction of travel, $\Omega \in \mathbb{S}_2$, speed  $c>0$ (and hence velocity $\upsilon = c\Omega$),  and neutron energy, $E\in \mathbb{R}$. It is not uncommon in the physics literature, as indeed we shall do here, to assume that energy is a function of velocity ($E = m|\upsilon|^2/2$), thereby reducing the number of variables by one. This allows us to describe the dependency of flux more simply in terms of time and, what we call, the {\it configuration variables} $ (r, \upsilon) \in  D \times V$ where $D\subseteq\mathbb{R}^3$ is a smooth, open,  connected and bounded domain of concern such that $\partial D$ has zero Lebesgue measure and 
$V$ is the velocity space, which can now be taken to be $V = \{v\in\mathbb{R}^3: \upsilon_{\texttt{min}}<|v|<\upsilon_{\texttt{max}}\}$, where $0<\upsilon_{\texttt{min}}<\upsilon_{\texttt{max}}<\infty$. 

\smallskip

Before stating the NTE, let us remind the reader of some elementary nuclear physics, which is required to  describe the evolution of neutron flux. In the most basic of flux models, there are essentially only four processes at the level of the atomic nuclei which contribute to the evolution of neutron flux. 

\smallskip

The first is {\it spontaneous neutron emission} from unstable nuclei. This comes from  radioactive isotopes whose nuclei are excited. They cause what is known as  {\it non-transmutation emissions}, in which a neutron is ejected with an escape velocity ({\it neutron emission}), or, conversely, what are called {\it transmutation emissions} in which the nucleus instantaneously fragments  into two or more nuclei ({\it spontaneous fission}) with a range of possible masses, emitting one or more neutrons with escape velocities in the process. 

\smallskip

The second process pertains to {\it neutron scattering}. This is where a neutron travelling with a given velocity passes in close proximity to an atomic nucleus, which, in our model,  results in an instantaneous change of  velocity.
\smallskip

The third process is {\it neutron-induced fission}. This is the classical setting in which a neutron travelling with a given velocity strikes an atomic nucleus sending it into an excited state, from which it instantaneously fragments into two or more nuclei, simultaneously releasing one or more neutrons.

\smallskip 

The fourth and final process is {\it neutron capture}. In this setting, a neutron travelling with a given velocity strikes an atomic nucleus, but instead of causing nuclear fission, it is absorbed into the nucleus. It can also be the case that neutrons decay into other subatomic particles, and thus disappear from the system. To all intents and purposes, we can treat this as neutron capture.

\smallskip 

When modelling the transmission of neutrons in a fissile material, those neutrons which have been released from nuclei are known as {\it prompt neutrons}.

\smallskip 

With more advanced modelling, one can also take account of the fact that some of the processes described above can also involve other types of nuclear emissions, often in addition to neutrons. 
 These include alpha and beta particles and gamma radiation. Whilst the former two are not sufficiently energetic to cause fission, sufficiently energetic gamma rays are able to induce fission.

\smallskip 

Spontaneous fission and neutron-induced fission can also produce what are known as {\it delayed neutrons}. These  are neutrons released from a fission product (isotope) some time after fission has occurred. In terms of modelling, they are spontaneous neutron emissions which occur at the site of neutron-induced fission but at a moment  later in time.  Delayed neutrons are only in a delayed state until they are released after which they are considered as prompt neutrons. 

\smallskip 

We refer to models which take account of the full range of flux profiles as multi-species models.

\section{Neutron Transport Equation}
Let us now write down the basic neutron transport equation (prompt neutrons only), which has been widely considered in a variety of physics and engineering literature (cf. \cite{DS, PP}, to name but two classical references), and somewhat more sporadically studied in the mathematical literature. See \cite{DL6, M-K, LPS} for the three most  authoritative mathematical texts in more recent times, as well as e.g. \cite{MWY, Hbook, MT} for some of the rarer examples of the probabilistic treatment of the NTE.

\smallskip

Neutron flux at time $t\geq  0$ is henceforth identified as $\Psi_t: {D}\times V\to [0,\infty)$, and the classical presentation of its evolution in time is given by the integro-differential equation, also known as  the {\it forward neutron transport equation}\footnote{Here and everywhere else in the document, $\nabla$ is the gradient operator with respect to the variable $r\in\mathbb{R}^3$.},   
\begin{align}
\frac{\partial}{\partial t}\Psi_t(r, \upsilon)&=- \upsilon\cdot\nabla\Psi_t(r, \upsilon)  -\sigma(r, \upsilon)\Psi_t(r, \upsilon) +Q(r,\upsilon, t)\notag\\
&
+\int_{V}\Psi_t(r, \upsilon') \sigma_{\texttt{s}}(r, \upsilon') \pi_{\texttt{s}}(r,  \upsilon', \upsilon)\d\upsilon' + \int_{V}\Psi_t(r, \upsilon')\sigma_{\texttt{f}}(r, \upsilon')  \pi_{\texttt{f}}(r, \upsilon', \upsilon)\d\upsilon',
\label{NTE}
\end{align}
where the different components (or {\it cross-sections as they are known in the physics literature}) are all uniformly bounded and measurable with the following interpretation:
\begin{align*}
\sigma_{\texttt{s}}(r, \upsilon') &: \text{ the rate at which scattering occurs from incoming velocity $\upsilon'$,}\\
\sigma_{\texttt{f}}(r, \upsilon') &: \text{  the rate at which fission occurs from incoming velocity $\upsilon'$,}\\
\sigma(r, \upsilon) &: \text{ the sum of the rates } \sigma_{\texttt{f}}+ \sigma_{\texttt{s}}, \text{ also known as the {\it total} cross section} \\
\pi_{\texttt{s}}(r, \upsilon', \upsilon)\d\upsilon' &: \text{  the scattering yield at velocity $\upsilon$ from incoming velocity }  \upsilon', \\
 &\hspace{0.5cm}\text{ satisfying }\textstyle{\int_{V}}\pi_{\texttt{s}}(r, \upsilon,  \upsilon')\d \upsilon'=1,\\
 \pi_{\texttt{f}}(r, \upsilon', \upsilon)\d\upsilon' &:  \text{  the neutron yield at velocity $\upsilon$ from fission with incoming velocity }   \upsilon',\\
 &\hspace{0.5cm}\text{ satisfying }\textstyle{\int_{V}}\pi_{\texttt{f}}(r, \upsilon,  \upsilon')\d \upsilon'<\infty, \text{ and }\\
 Q(r,\upsilon, t) &: \text{ non-negative source term. }
\end{align*}
It is normal to assume that  all quantities are uniformly bounded away from infinity.
It is also usual to assume the additional boundary conditions 
\begin{equation}
\left\{
\begin{array}{ll}
\Psi_0(r, \upsilon) = g(r, \upsilon) &\text{ for }r\in  D, \upsilon\in{V},
\\
&
\\
\Psi_t(r, \upsilon)= 0& \text{ for } t\ge 0 \text{ and } r\in \partial D
\text{ if }\upsilon
\cdot{\bf n}_r<0,
\end{array}
\right.
\label{BC}
\end{equation}
where  ${\bf n}_r$ is the outward facing normal of $D$ at $r\in \partial D$ and $g: D\times {V}\to [0,\infty)$ is a bounded, measurable function which we will later assume has some additional properties.  Roughly speaking, as the forward equation describes where  particles could have evolved from  in order to contribute to the current configuration, this boundary condition means that particles from outside the domain with incoming velocity are not taken into account.
 The second of the above two boundary condition is sometimes written $\Psi_t|_{\partial D^-} =0$, where $\partial D^- = \{(r,\upsilon)\in \partial D\times V:\upsilon
\cdot{\bf n}_r<0 \}$.
It is also usual to set $Q =0$ when considering a rector with a multiplying medium,
 as the resulting fission will overwhelm the radioactive source term.

\smallskip

The notion of a solution of the form \eqref{NTE} turns out to be too strong to expect to make mathematical sense of it. This is predominantly due to the non-diffusive nature of the equation, in particular the non-local nature of the scattering and fission operators as well as regularity issues on the domain $D\times V$ in relation to continuity properties of e.g. the operator $\upsilon\cdot\nabla$. It is much more natural to look for solutions that belong to e.g. an appropriate $L_2$ space.  This is, moreover, helpful when looking to understand \eqref{NTE} as a {\it backwards} equation, rather than a {\it forwards} equation.
\smallskip

 With some rearrangements, the components of \eqref{NTE} separate into transport, scattering and fission. 
  Specifically,  
\begin{equation}
\left\{
\begin{array}{rll}
{\fT}{g}(r, \upsilon)  &:= - \upsilon\cdot\nabla{g}(r, \upsilon)   - \sigma(r,\upsilon){g}(r, \upsilon) &\text{ (forwards transport) }\\
&\\
{\fS}{g}(r, \upsilon)  &:= \int_{V}{g}(r, \upsilon') \sigma_{\texttt{s}}(r, \upsilon)\pi_{\texttt{s}}(r, \upsilon', \upsilon)\d\upsilon'  &\text{ (forwards scattering) }\\
&\\
{\fF}{g}(r, \upsilon) &: =  \int_{V}{g}(r, \upsilon') \sigma_{\texttt{f}}(r, \upsilon) \pi_{\texttt{f}}(r, \upsilon', \upsilon)\d\upsilon'&\text{ (forwards fission)}
\end{array}
\right.
\label{forwards_operators}
\end{equation}
such that all operators are defined on $D\times V$ and their action is zero otherwise.
Let us momentarily consider the operator on the right-hand side of \eqref{NTE} as acting on $L_2(D\times V)$, the space of square integrable functions on $D\times V$, and write 
\[
\langle f, g\rangle =  \int_{  D\times V}f(r,\upsilon)g(r,\upsilon)\d r\d\upsilon
\]
for the associated inner product.  Note that, for $f,g\in L_2(D\times V)$ such that both $\upsilon\cdot\nabla f$ and $\upsilon\cdot\nabla g$ are well defined as distributional derivatives, which are also  in the space $ L_2(D\times V)$, with $g$ respecting the second of the boundary conditions in \eqref{BC}, we can verify with a simple integration by parts that, for $\upsilon\in V$,
\begin{align} 
\langle f, \upsilon\cdot\nabla g \rangle &= \int_{\partial D\times V} (\upsilon\cdot \upsilon') f(r,\upsilon')g(r,\upsilon') \d r\d \upsilon'  -\langle  \upsilon\cdot\nabla  f,g \rangle=  -\langle  \upsilon\cdot\nabla  f,g \rangle
\label{nabladual}
\end{align}
 {\it providing} we insist that $f$ respects the boundary $f (r, \upsilon)= 0$ for $r\in \partial D$
 if $\upsilon
\cdot{\bf n}_r>0$. Moreover, Fubini's theorem also tells us that, for example,  with $f, g\in L_2(D\times V)$,
\begin{align*}
\langle f, \int_{V} g(\cdot, \upsilon')\sigma_{\texttt{s}}(\cdot,\upsilon')\pi_{\texttt{s}}(\cdot, \upsilon', \cdot)\d\upsilon' \rangle 
&=\int_{{D}\times V\times V} f(r,\upsilon)\sigma_{\texttt{s}}(r,\upsilon')g(r, \upsilon')\pi_{\texttt{s}}(r, \upsilon', \upsilon)\d\upsilon'  \d r \d \upsilon \\
&=\int_{{D}\times V}\sigma_{\texttt{s}}(r,\upsilon') \int_V f(r,\upsilon)\pi_{\texttt{s}}(r, \upsilon', \upsilon)  \d \upsilon  \,  g(r, \upsilon')\d r\d\upsilon' \\
&=\langle
\sigma_{\texttt{s}}(\cdot, \cdot) \int_V f(\cdot,\upsilon)\pi_{\texttt{s}}(\cdot, \cdot, \upsilon)  \d \upsilon, g
\rangle.
\end{align*}
These computations tell us that, for $f, g \in L_2({D}\times V)$ such $\upsilon\cdot\nabla g$ and  $\upsilon\cdot\nabla f$ are well defined in the distributional sense and, moreover, that $g(r, \upsilon)= 0$ for $r\in \partial D$  if $\upsilon
\cdot{\bf n}_r<0$, and for $f \in L_2({D}\times V)$ such that $f (r, \upsilon)= 0$ for $r\in \partial D$  if $\upsilon
\cdot{\bf n}_r>0$, 
 \[
 \langle f, (\fT + \fS + \fF )g\rangle  =  \langle (\bT + \bS + \bF ) f, g\rangle,
 \]
where now we identify the transport, scattering and fission operators as
\begin{equation}
\left\{
\begin{array}{rll}
{\bT}{f}(r, \upsilon)  &:=  \upsilon\cdot\nabla{f}(r, \upsilon)   &\text{ (backwards transport) }\\
&\\
{\bS}{f}(r, \upsilon) &:= \sigma_{\texttt{s}}(r, \upsilon)\int_{V}{f}(r, \upsilon') \pi_{\texttt{s}}(r, \upsilon, \upsilon') \d\upsilon'  - \sigma_{\texttt{s}}(r, \upsilon){f}(r, \upsilon) &\text{ (backwards scattering) }\\
&\\
{\bF}{f}(r, \upsilon) &: =  \sigma_{\texttt{f}}(r, \upsilon) \int_{V}{f}(r, \upsilon') \pi_{\texttt{f}}(r, \upsilon, \upsilon')\d\upsilon' -\sigma_{\texttt{f}}(r, \upsilon)f (r,\upsilon)&\text{ (backwards fission)}
\end{array}
\right.
\label{backwards_operators}
\end{equation}
such that all operators are defined on $D\times V$ with zero action otherwise.
The reader will immediately note that, although the terms in the sum ${\bT} + {\bS} + {\bF}$ are identifiable as the adjoint of the terms in the sum ${\fT} + {\fS} + {\fF}$, 
the same can not be said for the individual `\texttt{T}', `\texttt{S}' and `\texttt{F}' operators. That is to say, the way we have grouped the terms does not allow us to say that e.g. ${\bT}$ is the adjoint operator to ${\fT}$ and so on.  

\smallskip

The reason for this difference in grouping of terms lies with how one reads the operators in terms of infinitesimal generators as a probabilist. Although this will not make any difference in the analysis of this paper, we keep to this notation for the sake of consistency with further related articles which offer a probabilistic perspective on the backwards NTE; see \cite{MCNTE, SNTE, SNTE2}.
\smallskip

Roughly speaking, ${\bT}$, with an appropriately defined domain, 
 is the generator of the rather simple Markov process consisting of a deterministic motion with velocity $\upsilon$, i.e. transport due to pure advection, with killing on exiting the domain $D$. Similarly, with an appropriately defined domain, the operator ${\bS}$ is the generator corresponding to scattering,  in which a particle travelling with velocity $\upsilon$ at position $r$ is removed at rate $\sigma_{\texttt{s}}$ and replaced by a new particle at $r$ with velocity $\upsilon'$ chosen with probability $\pi_{\texttt s}(r,\upsilon,\upsilon')\d\upsilon'$. Taking advantage of the fact that $\textstyle{
\int_V\pi_{\texttt{s}}(r,\upsilon, \d\upsilon')\d\upsilon'=1}$ we can also  write 
\[
\sigma_{\texttt{s}}(r, \upsilon)\int_{V}{f}(r, \upsilon')\pi_{\texttt{s}}(r, \upsilon, \upsilon') \d\upsilon'  -\sigma_{\texttt{s}}(r, \upsilon) {f}(r, \upsilon) =\sigma_{\texttt{s}}(r, \upsilon)\int_{V}[{f}(r, \upsilon') - {f}(r, \upsilon) ]\pi_{\texttt{s}}(r, \upsilon, \upsilon') \d\upsilon' 
\]
and also note that it takes the classical form of a difference operator. Finally ${\bF}$ is the generator action of a fission even in which a particle travelling with velocity $\upsilon$ at position $r$ is removed at rate $\sigma_{\texttt{f}}$ and replaced by an 
average number of particles $\pi_{\texttt{f}}(r,\upsilon,\upsilon')\d\upsilon'$ moving onwards from $r$ with velocity $\upsilon'$.

\smallskip

This leads us to the so called {\it backwards neutron transport equation} (which is also known as the {\it adjoint neutron transport equation}) given by 
\begin{align}
\frac{\partial}{\partial t}\psi_t(r, \upsilon) &=
\upsilon\cdot\nabla\psi_t(r, \upsilon)  -\sigma(r, \upsilon)\psi_t(r, \upsilon)\notag\\
&\hspace{1cm}+ \sigma_{\texttt{s}}(r, \upsilon)\int_{V}\psi_t(r, \upsilon') \pi_{\texttt{s}}(r, \upsilon, \upsilon')\d\upsilon' + \sigma_{\texttt{f}}(r, \upsilon) \int_{V}\psi_t(r, \upsilon') \pi_{\texttt{f}}(r, \upsilon, \upsilon')\d\upsilon',
\label{bNTE}
\end{align}
with additional boundary conditions 
\begin{equation}
\left\{
\begin{array}{ll}
\psi_0(r, \upsilon) = g(r, \upsilon) &\text{ for }r\in D, \upsilon\in{V},
\\
&
\\
\psi_t(r, \upsilon) = 0& \text{ for } t \ge 0 \text{ and } r\in \partial D
\text{ if }\upsilon
\cdot{\bf n}_r>0.
\end{array}
\right.
\label{BC1}
\end{equation}
Similarly to previously, the second of these two conditions is often written $\psi_t|_{\partial D^+} = 0$, where $\partial D^+ : = \{(r,\upsilon)\in \partial D\times V :
\upsilon
\cdot{\bf n}_r>0 \}$.

\smallskip

 The NTE has played  a prominent role in real-world modelling and, for many years, has found a home in commercial software which is used in the nuclear safety industry. In particular, this is most prominent in the modelling and design of environments which are exposed to radioactive material, from nuclear reactor cores and hospital equipment, through to equipment used to irradiate produce that is sold in supermarkets, thereby prolonging its shelf-life. More recently, with the notion of human interplanetary space exploration becoming less of a sci-fi fantasy and more of a fast approaching  reality, an understanding of how long-lasting and compact nuclear power sources, for e.g. Moon or Mars bases has become increasingly important.

\smallskip 

Figure \ref{reactorcore} below  illustrates a typical geometrical model of a reactor core rod, cladding and outer shielding.\footnote{The authors are grateful to Prof. Paul Smith from Wood who has given us permission to use these images which were constructed with Wood nuclear software ANSWERS.} The structural design of such a reactor  can easily be stored as virtual environment (i.e. storing the coordinates of the different geometrical domains and the material properties in each domain) with around 150MB of data, on to which extensive  data libraries of numerical values for the respective quantities $ \sigma_{\texttt{s}},  \sigma_{\texttt{f}},  \pi_{\texttt{s}},  \pi_{\texttt{f}}$  can be mapped. (It is an otherwise little known fact that countries which are heavily invested in nuclear power, such as the UK, USA, France, China, etc., are all in possession of such numerical libraries of cross sections, which have been carefully built up over decades.)

\begin{figure}[h]
\includegraphics[width= 0.45\textwidth]{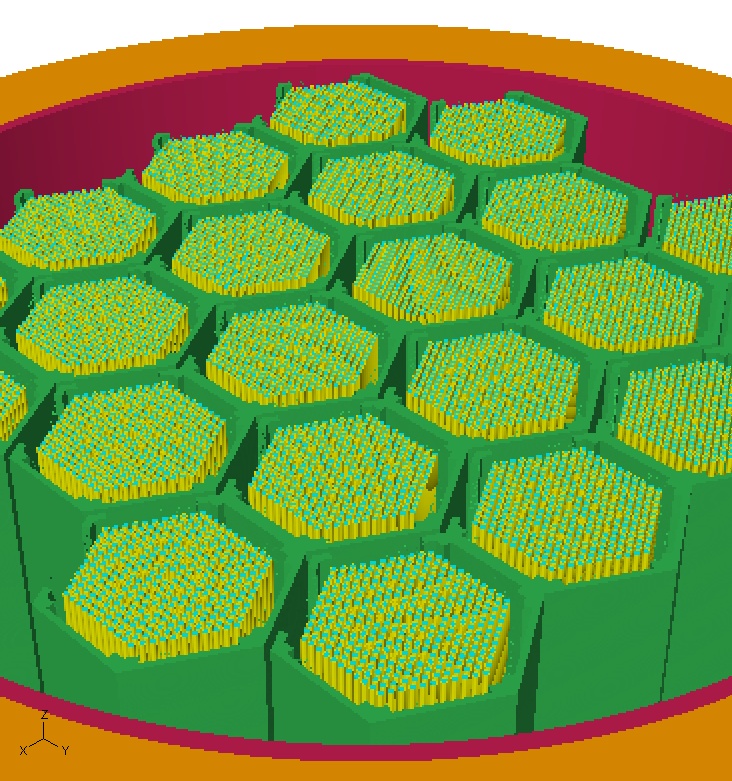}
$\mbox{ }$
\includegraphics[width= 0.45\textwidth]{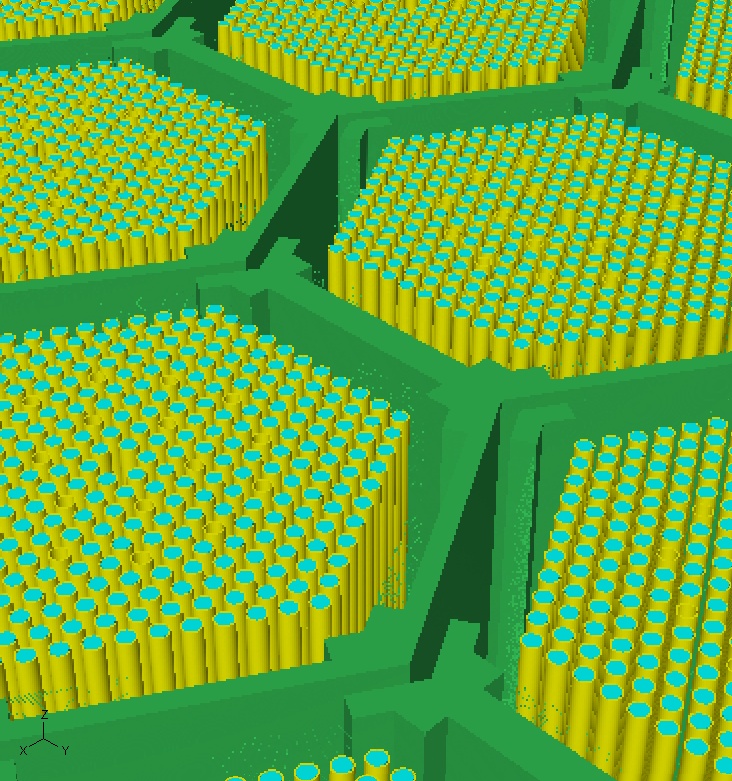}
\caption{\rm A virtual model of a nuclear reactor core with colour indicating the respective fissile properties of the virtual materials used. Uranium rods are arranged into hexagonal cells which are arranged within a larger containment casing.}
\label{reactorcore}
\end{figure}

\smallskip

One of the  principal ways in  which neutron flux is understood is to look for the leading eigenvalue and associated ground state eigenfunction. Roughly speaking, this means looking for an associated triple of eigenvalue $\lambda\in\mathbb{R}$,  non-negative right eigenfunction  $\varphi: {D}\times V\to[0,\infty)$ in $L_2(D\times V)$ satisfying $\varphi|_{\partial D^+} =0$ and  a non-negative left eigenfunction $\tilde\varphi$ on ${D}\times V$  in $L_2(D\times V)$ satisfying $\tilde\varphi|_{\partial D^-} = 0$ such that 
\[
\lambda\langle\varphi,  f
\rangle =\langle ( {\bT}+ {\bS} + {\bF})\varphi, f  \rangle \quad \text{ and } \quad\lambda \langle f, \tilde{\varphi} \rangle
 =\langle ( {\bT}+ {\bS} + {\bF})f, \tilde{\varphi} \rangle.
\]
As such, this introduces the notion of {\it fissile stability}, in particular in the case that $\lambda = 0$. This is naturally the desired scenario\footnote{Strictly speaking the reality is that, nuclear reactors are kept in a slightly supercritical state. The reason for this is that at criticality, as proved in \cite{SNTE2}, neutron activity will eventually die out.} for a nuclear reactor.
\smallskip

In the physics literature, it is thus often understood that, to leading order, the NTE \eqref{bNTE} is solved in the approximate sense
\begin{equation}
\psi_t(r,\upsilon) = {\rm e}^{\lambda t}\langle g, \tilde{\varphi}\rangle \varphi(r, \upsilon) + o({\rm e}^{\lambda t}), \qquad t\geq 0.
\label{sim}
\end{equation}

Note that  the scenario that $\lambda>0$ is obviously to be avoided  in practice as this would correspond to a set-up that could result in exponential growth in fission.

\smallskip

The approximation \eqref{sim} can be seen as a functional version of the Perron-Frobenius Theorem and  has given rise to a number of different numerical methods for estimating the value of the eigenvalue $\lambda$ as well as the eigenfunctions $\varphi$ and $\tilde{\varphi}$. One approach pertains to the discretisation of \eqref{NTE} followed by the use numerical analytic methods; see \cite{SG}. Another pertains to the previously alluded to identification of the solution to the NTE as the linear semigroup of a Markov branching process, which in turn implies Monte Carlo methods involving the simulation of the aforesaid branching process. Such methods are computationally expensive, as branching processes, being tree-like structures, are complex to simulate, e.g. from the point of view of parallelisation. 
In  related papers to this one, we will discuss a new Monte Carlo approach to the NTE based on some of the stochastic analysis we deal with in this article as well as in related work undertaken by the authors of this paper; see \cite{SNTE, SNTE2, MCNTE}.

\smallskip

The aim of this paper is manifold. First and foremost, we aim to reposition the theory  of the NTE into a contemporary probabilistic setting. We will do this by explaining a precise relationship between the NTE and a two different families of Markov processes via Feynman--Kac type formulae. Indeed, this article is one of a cluster of forthcoming pieces of work, which take a new and predominantly probabilistic point of view of the NTE; cf \cite{MCNTE, SNTE, SNTE2}. Next  we want  to introduce the notion of the (multi-species) NTE into the  literature, which generalises \eqref{NTE} by simultaneously modeling the flux of all species of particles and radiation involved in the process of nuclear fission. In doing so we will show that, just as in the classical setting, one may develop the notion of a lead eigenvalue and eigenfunction, which is an important part of describing fissile stability. As such, the current article is part review of existing theory and part presentation of new research results based on generalisation of existing results

\smallskip

Together with the accompanying papers \cite{MCNTE, SNTE, SNTE2}, we believe that the probabilistic perspective presented here, i.e. coupling the solutions to the  NTE with averaging procedures of certain Markov processes, opens up the possibility of  many questions that can be considered at depth in the arena of stochastic analysis and Monte Carlo algorithms, which are currently missing from the literature.  Indeed, returning to the kind of environments seen in Figure \ref{reactorcore}, there are many questions concerning how to analyse and numerically generate the leading eigenfunctions and eigenvalue to a reasonable  degree of precision. Such questions might include: {\it What is the connection of the eigendecomposition discussed in this paper and e.g. $R$-theory or the theory of general Harris recurrence for stochastic processes (cf. \cite{DMT, T3, T2})? }{\it How do different stochastic representations lead to different Monte Carlo simulations?}  {\it Based on stochastic representation how does one measure convergence of Monte Carlo algorithms? How strong can they be predicted to be? What kind of variance reduction techniques does stochastic representation suggest?} {\it Does the inclusion of multi-species models make estimation of the leading eigenvalue more accurate?}

\smallskip

\section{Organisation of the paper}  In the next section, we give a brief overview of the key mathematical literature for the NTE. (Note we do not stray beyond mathematical literature, as the physics and engineering literature is significantly more expansive.) Thereafter in Section \ref{MNTEsect}, we introduce the multi-species NTE (MNTE) and its rigorous formulation, existence, uniqueness and asymptotics in the setting of an abstract Cauchy problem. In particular, we show how the unique solution is identified as a $c_0$-semigroup in the appropriate $L_2$ space. In Section \ref{MNBPsect}, we introduce a spatial branching process that is constructed using the cross sections that appear in the  NTE to describe its stochastic evolution.  Here we introduce its expectation semigroup.  In Section \ref{RW}, we provide a second stochastic representation to the expectation semigroup introduced in the previous section via a classical method of the many-to-one formula.
\smallskip

Ideally, we would like to claim that the expectation semigroup discussed in Sections \ref{MNBPsect} and \ref{RW} agree with the $c_0$-semigroup introduced in Section \ref{MNTEsect} (its formal definition appearing just above Theorem \ref{ACPsoln}). This is particularly desired as  it  forms the foundations of how Monte Carlo simulation of the physical process can be used to develop a numerical solution to the MNTE. In Section \ref{consolidate}, we consolidate the two notions of semigroup and show that there is partial agreement in an appropriate sense. As far as we are aware, this is a point which is currently not clearly discussed in the literature. Finally we end the paper with a proof of one of the main theorems in Section \ref{MNBPsect} which provides the asymptotic behaviour of the solution to the MNTE in terms of the lead eigenfunction. This is a new result in the multi-species setting in the sense that we have allowed for multiple types of prompt emissions (both particles and radioactive emissions) rather than the case of only one type of prompt emission dealt with in \cite{M-K}; we also allow for multiple types of delayed emissions (that is, emissions that are pre-emptively held in an unstable radioactive isotope product from an  earlier fission event). Our proof nonetheless takes inspiration from the classical approach of \cite{DL6, M-K}, and remains loyal to the techniques there.

\section{Historical remarks on the mathematical treatment of the NTE} 
Classical texts such as Davison and Sykes \cite{DS} were once hailed as a bible of mathematical knowledge during the 1950s post {\it Manhattan Project} era when rapid technological advances lead to  the construction of the very first nuclear reactors driving commercial power stations. Around this time, there was an understanding of how to treat the NTE in special geometries and also by imposing an isotropic scattering and fission, see for example Lehner \cite{Leh} and Lehner and Wing \cite{LW1, LW2}. It was also understood quite early on that the natural way to cite the NTE is via the linear differential transport equation associated to a suitably defined operator on a Banach space. Moreover, it was understood that in this formulation, a spectral decomposition should play a key role in representing solutions, see e.g. J\"orgens \cite{J}, Pazy and Rabinowitz \cite{PR69}. 
This notion was promoted by the work of R. Dautray and collaborators, who showed how  $c_0$-semigroups form a natural framework within which one may analyse the existence and uniqueness of solutions to the NTE; see  \cite{D} and \cite{DL6}. Moreover, a similar approach has also been pioneered by Mokhtar-Kharroubi \cite{M-K}.\smallskip

The probabilistic interpretation of the NTE was appreciated from the very first treatments of the NTE (see e.g. \cite{DS} and references therein as well as Bell \cite{Bell}). Indeed, the physical description of nuclear fission, when governed by basic principles, allowing for  additional randomness,  is nothing more than a branching Markov process. Numerous derivations of the NTE from this perspective can be found in the literature to various degrees of rigour; see e.g.  Bell \cite{Bell}, Mori et al. \cite{MWY}, Pazy and Rabinowitz, \cite{PR}, Lewins \cite{L} and P\'azsit and P\'al. \cite{PP}.
\smallskip

A more modern treatment of the probabilistic representation through Feynman-Kac expectation semigroups and the connection to the theory of Markov diffusions is found in Dautray et al. \cite{D}. A purely  probabilistic can be found in Lapeyre et al. \cite{LPS}. See also the accompanying papers to this one \cite{MCNTE, SNTE, SNTE2}.
\smallskip

We finish this section by noting that there is a body of literature that pertains to the numerical analysis of the NTE. Recent work in this field, including the notion of uncertainty quantification, can be found in e.g. \cite{MT, GPS, SG}. See also references therein.

\section{Multi-species (Backwards) Neutron Transport Equation}\label{MNTEsect} In the following discussion, rather than talk about typed particles, we prefer to say typed `emissions' as the different types correspond to particles, electromagnetic rays (e.g. gamma rays) and isotopes (which are considered to be carriers for delayed emissions).
Let us now introduce an advanced version of the NTE, which takes account of both  non-transmutation emissions as well as  transmutation emissions,  in particular, allowing for the inclusion of all types of emissions, prompt neutrons, delayed neutrons, alpha, beta and gamma emissions etc.  An important feature (and arguably a restriction) of our model is that only prompt neutrons can produce delayed emissions.
\smallskip

In order to keep track of the various emission types, we define the type space $I \coloneqq \{1, \dots, m\}$ for some $m \in \N$, ordered such that 
\begin{align*}
\text{type $1$ emissions:} & \text{ prompt neutrons (neutrons released immediately after fission)}\\
\text{types $2, \dots,\ell$ emissions:}& \text{ other prompt emissions (e.g. alpha, beta,  gamma emissions)}\\
\text{types $\ell+1, \dots, m$ emissions:} &\text{ isotopes (holding types/precursors) that hold delayed emissions.} 
\end{align*}

Finally, the set of admissible velocities for each of the types $i$ can be embedded within a common space $V = \{\upsilon\in \mathbb{R}^3: \upsilon_{\texttt{min}}\leq |\upsilon|\leq  \upsilon_{\texttt{max}}\}$, with $0 < \upsilon_{\texttt{min}} \le \upsilon_{\texttt{max}}< \infty$). We now consider the flux, $\psi_t(i,r, \upsilon) $ of type $i$ emissions through a given region $r\in D$  with velocity $\upsilon\in V$ at time $t\geq 0$. 
We are interested in the so called {\it multi-species neutron transport equation (MNTE)} which takes the form
\begin{align}
\frac{\partial}{\partial t}\psi_t(i,r, \upsilon)&= \upsilon\cdot\nabla
\psi_t(i,r, \upsilon)  -\sigma^i(r, \upsilon)\psi_t(i,r, \upsilon) \notag\\
&\hspace{1cm}
+ \sigma^i_{\texttt{s}}(r, \upsilon) \int_{V}\psi_t(i,r, \upsilon)\pi^i_{\texttt{s}}(r,  \upsilon, \upsilon')\d\upsilon'\notag\\ 
&\hspace{2cm}+ \sigma^i_{\texttt{f}}(r, \upsilon)\sum_{j = 1}^\ell\int_{V}\psi_t(j,r, \upsilon)  \pi^{i,j}_{\texttt{f}}(r, \upsilon, \upsilon')\d\upsilon' \notag\\
&\hspace{3cm} +\mathbf{1}_{(i=1)}\sigma^1_{\texttt{f}}(r, \upsilon)\sum_{j=\ell+1}^mm^j(r, \upsilon)\psi_t(j,r, \upsilon), \label{promptNTE}
\end{align}
for prompt emissions $i = 1,\cdots, \ell$, whereas, in the case of delayed emissions, $i =\ell+1, \cdots, m$ satisfies
\begin{align}
\frac{\partial}{\partial t}\psi_t(i,r, \upsilon)& = -\lambda_i\psi_t(i,r, \upsilon) + \lambda_i\sum_{j=1}^\ell\int_{V}\psi_t(j,r, \upsilon)\pi_{\texttt{f}}^{i,j}(r, \upsilon, \upsilon')\d\upsilon' \label{delayNTE},
\end{align}
which is of a simple form because it describes only how these emissions are held in a suspended state (no advection) before being converted back to prompt emissions. Similarly to before, have the following interpretation: 
\begin{align*}
\sigma^i_{\texttt{s}}(r, \upsilon) &: \text{ the rate at which scattering occurs for a type $i$ emission with incoming}\\
&\hspace{0.5cm} \text{velocity $\upsilon$,}\\
\sigma^i_{\texttt{f}}(r, \upsilon) &: \text{  the rate at which fission occurs for a type $i$ emission with incoming}\\
&\hspace{0.5cm}\text{velocity } \upsilon,\\
\sigma^i(r, \upsilon) &: \text{ the sum of the rates } \sigma^i_{\texttt{f}}+ \sigma^i_{\texttt{s}} \text{ and is known as the total cross section for a }\\
&\hspace{0.5cm}\text{type $i$ emission,}\\
\pi^i_{\texttt{s}}(r, \upsilon, \upsilon')\d\upsilon' &: \text{  the scattering yield at velocity  $\upsilon'$ from incoming velocity $\upsilon$ for a type $i$} \\
 &\hspace{0.5cm}\text{emission, satisfying }\textstyle{\int_V}\pi^i_{\texttt{s}}(r, \upsilon, \upsilon')\d\upsilon'=1,\\
 \pi^{i,j}_{\texttt{f}}(r, \upsilon, \upsilon')\d\upsilon' &:  \text{  the average type $j$ yield at velocity $\upsilon'$ from fission with incoming velocity}\\
 &\hspace{0.5cm}\mbox{$\upsilon$ for a type $i$ emission satisfying $\sum_{j = 1}^\ell 
\textstyle{\int_V}\pi^{i,j}_{\texttt{f}}(r, \upsilon, \upsilon')\d\upsilon' <\infty$,}\\
 m^j(r, \upsilon) &: \text{ the average type $j$ (unstable) isotope yield from a fission event due to a} \\
 &\hspace{0.5cm}\text{type 1 particle with incoming velocity $\upsilon$},\\
\lambda_i &: \text{ the decay rate for a type $i$ isotope.}
\end{align*}
There are a number of assumptions about the many cross sections that appear in the above equations that will remain in force throughout the remainder of this text. 
\begin{assumptions}\rm
All cross sections are non-negative, measurable and uniformly bounded from above. Moreover, all prompt emissions scatter and hence,  without loss of generality, we also assume that for for each $i=1,\cdots, \ell$, the terms $\sigma^i_{\texttt{s}}\pi^i_{\texttt{s}}$ are uniformly bounded away from the origin on $D\times V$. We need not assume that the cross sections $\sigma^i_{\texttt{f}}\pi^{i,j}_{\texttt{f}}$ are uniformly bounded away from the origin for $1\leq i,j\leq \ell$, with the exception of $i = 1$, for which it only makes sense that $\sigma^1_{\texttt{f}}m^j$ is uniformly bounded away from 0 for each $j = \ell+1, \cdots, m.$  Without loss of generality, we can assume that $0<\lambda_{\ell + 1}<\cdots <\lambda_m$.
\end{assumptions}

We also assume similar boundary conditions to the single-type case in the sense that emissions exiting the physical domain $D$ are killed. That is to say

\begin{equation}
\left\{
\begin{array}{ll}
\psi_0(i, r, \upsilon) = g(i, r, \upsilon) &\text{ for } 1\leq i\leq m, r\in D, \upsilon\in{V},
\\
&
\\
\psi_t(i,r, \upsilon) = 0& \text{ for }1\leq i\leq \ell, r\in \partial D
\text{ if }\upsilon
\cdot{\bf n}_r>0.
\end{array}
\right.
\label{MBC}
\end{equation}
For the second condition, we will write $\psi_t|_{\partial D^+} = 0$, where $\partial D^+ = \{(i, r,\upsilon) \in \{1,\cdots, \ell\}\times\partial D\times V:\upsilon\cdot{\bf n}_r>0\}$

\smallskip

Classical literature suggests that one can integrate delayed neutrons into the setting of the NTE by adding an inhomogeneity corresponding to the integral of incoming delayed neutrons from time $-\infty$ to the present; see e.g. \cite{DS}. A vectorial representation such as the one above can be found, however, in the work of  \cite{M-K}. There, only one category of prompt emissions are considered with multiple species of   delayed neutrons. 

\smallskip

As before, let us define the multi-species backward transport, scattering and fission operators as they appear in MNTE \eqref{promptNTE} and \eqref{delayNTE}, acting on $f\in \prod_{i=1}^m L_2(D\times V)$, so that, for $i = 1,\cdots m$,
\begin{equation*}
\left\{
\begin{array}{rl}
{\bT}_i{f}(\cdot, r, \upsilon)  &:= \mathbf{1}_{(1\leq i \leq  \ell)} \upsilon\cdot\nabla f(i, r, \upsilon) \\
&\\
{\bS}_i{f}(\cdot, r, \upsilon)  &:=\mathbf{1}_{(1\leq i \leq  \ell)}\int_{V}
[f(i, r, \upsilon') - f(i, r, \upsilon)]\sigma^i_{\texttt{s}}(r, \upsilon) \pi^i_{\texttt{s}}(r,  \upsilon, \upsilon')\d\upsilon'  \\
&\\
{\bF}_i{f}(\cdot, r, \upsilon) &: =   \mathbf{1}_{(1\leq i\leq \ell)}\left(\displaystyle\sum_{j = 1}^\ell\int_{V}f(j, r, \upsilon')\sigma^i_{\texttt{f}}(r, \upsilon)  \pi^{i,j}_{\texttt{f}}(r, \upsilon, \upsilon')\d\upsilon' - \sigma^i_{\texttt{f}}(r, \upsilon)f(i, r, \upsilon')\right) \\
&\hspace{1cm}+\mathbf{1}_{(i=1)}\displaystyle\sum_{j=\ell+1}^m\sigma^i_{\texttt{f}}(r, \upsilon)m^j(r, \upsilon) f(j, r, \upsilon)\\
&
\hspace{2cm}+\mathbf{1}_{(\ell+1\leq i\leq m)}\left(\lambda_i\displaystyle\sum_{j=1}^\ell\int_{V} f(j, r, \upsilon')\pi_{\texttt{f}}^{i,j}(r, \upsilon, \upsilon')\d\upsilon'  -\lambda_if(i, r, \upsilon) \right),
\end{array}
\right.
\label{Mbackwards_operators}
\end{equation*}
with zero action otherwise.
\smallskip

It is not often that  MNTE is stated as above in \eqref{promptNTE} and \eqref{delayNTE} in existing literature; see e.g.  \cite{M-K} for presentation of the NTE in a similar vectorial format, which allows for only one category of prompt neutrons.
\smallskip

The requirement that all cross sections are uniformly bounded is by far not the weakest assumption we can make (see e.g. Chapter XXI of \cite{DL6}). 
\smallskip

The precise mathematical sense in which we must understand solutions to the coupled system \eqref{promptNTE} and \eqref{delayNTE} needs some discussion before we can proceed. 
To this end, we shall first introduce some notational conventions.

\smallskip

As alluded to above, we are interested in an vector space of functions, written as the column vector $g(\cdot) = (g(1,\cdot), \dots, g(m,\cdot))^{\texttt T}$, whose entries   $g(i,\cdot): D\times  V\to[0,\infty)$, for each $i  =1,\cdots, m$. More precisely we are interested in functions $f\in \prod_{j= 1}^m L_2({D}\times V)$, which is easily verified to be itself an $L_2$ space with inner product given by 
\begin{equation}
\langle f, g \rangle = \sum_{i = 1}^m (f, g)_i,
\quad\text{ 
where 
 }\quad
(f,g)_i = \int_{D\times V} f(i, r, \upsilon) g(i, r, \upsilon) \d r\d \upsilon.
\label{innerprods}
\end{equation}
Generally speaking, for a scalar quantity which is indexed by $i$, say $a(i)$, when written without the index, we will understand it to be a column vector. Sometimes we will want to put $f\in \prod_{j= 1}^m L_2({D}\times V)$ on the diagonal of an $m\times m$ matrix, in which case we will write $\diagonal(f)$. For our transport, scattering and fission operators, we will understand $\bT = \diagonal(
\bT_1, \cdots, \bT_m)$, however,  we will understand  $\bF$ to be the matrix acting on vectors $f\in \prod_{j= 1}^m L_2({D}\times V)$, with  $i,j$-th entry given by 
\begin{align*}
{\bF}_{i,j}{f}(j, r, \upsilon) &: =   \mathbf{1}_{(1\leq i,j\leq \ell)}\left(\displaystyle \int_{V}f(j, r, \upsilon')\sigma^i_{\texttt{f}}(r, \upsilon)  \pi^{i,j}_{\texttt{f}}(r, \upsilon, \upsilon')\d\upsilon' - \mathbf{1}_{(i = j)} \sigma^i_{\texttt{f}}(r, \upsilon)f(i, r, \upsilon')\right) \\
&\hspace{1cm}+\mathbf{1}_{(i=1, \ell + 1\leq j\leq m)} \sigma^i_{\texttt{f}}(r, \upsilon)m^j(r, \upsilon) f(j, r, \upsilon)\\
&
\hspace{2cm}+\mathbf{1}_{(\ell+1\leq i\leq m, 1\leq j\leq \ell)}\left(\lambda_i\displaystyle \int_{V} f(j, r, \upsilon')\pi_{\texttt{f}}^{i,j}(r, \upsilon, \upsilon')\d\upsilon'  -\mathbf{1}_{(i = j)} \lambda_if(i, r, \upsilon) \right).
\end{align*}
The operator $\bS$ can be handled similarly. 

\smallskip

We are fundamentally interested in a {\it classical solution} to the so-called (initial-value) {\it abstract Cauchy problem} (ACP)
\begin{equation}
\left\{
\begin{array}{rl}
\dfrac{\partial}{\partial t}u_t &= ({\bT} + {\bS}+{\bF})u_t
\\
u_0& = g 
\end{array}
\right.
\label{ACP}
\end{equation}
where $u_t$ is treated as a column vector belonging  to   the space $\prod_{j= 1}^m L_2({D}\times V)$, for $t\geq 0$.   Specifically this means that, $(u_t, t\geq 0)$ is continuously differentiable in this space. In other words,   there exists a $\dot{\psi}_t\in \prod_{j= 1}^m L_2({D}\times V)$, which is time-continuous in $\prod_{j= 1}^m L_2({D}\times V)$ with respect to $\norm{\cdot}_2$, such that  $\lim_{h\to 0}h^{-1}\norm{
u_{t+h} - u_t}_2= \dot{\psi}_t$ for all $t\geq 0$.

\smallskip

The theory of $c_0$-semigroups gives us a straightforward approach to describing the unique solution to \eqref{ACP}.
Recall that a $c_0$-semigroup  also goes by the name of a strongly continuous semigroup and, in the present context, this means a family of time-indexed operators, $(\texttt{V}_t, t\geq0)$, on $\prod_{j= 1}^m L_2({D}\times V)$ with the properties that \begin{itemize}
\item[(i)] $\texttt{V}_0 = {\rm Id}$, 
\item[(ii)] $\texttt{V}_{t+s}[g] = \texttt{V}_t[\texttt{V}_s[g]]$, for all $s, t\geq 0$, $g\in \prod_{j= 1}^m L_2({D}\times V)$ and \item[(iii)] for all $g\in \prod_{j= 1}^m L_2({D}\times V)$, $\lim_{h\to0}\norm{\texttt{V}_h[g] - g}=0$.
\end{itemize}

To see how $c_0$-semigroups relate to \eqref{ACP}, let us  define ${\bA}: = {\bT} + {\bS}+{\bF}$ and define $(\texttt{V}_t[g], t\geq 0)$ the semigroup generated by ${\bA}$ via
\begin{equation}
\label{V}
\texttt{V}_t[g] := \exp(t {\bA})g, \qquad g\in \prod_{j= 1}^m L_2({D}\times V).
\end{equation} Note that 
\[
{\rm Dom}({\bA}): = \left\{g\in \prod_{j= 1}^m L_2({D}\times V) : \lim_{h\to 0}h^{-1}\norm{\texttt{V}_h[g] - g}_2 \text{ exists}
\right\}
\]
is the  domain of $\bA$ and standard theory  (cf. \cite{EN}) tells us that  $\texttt{V}_t[g]\in {\rm Dom}({\bA})$ for all $t\geq 0$, with $g$ as above.
Proposition II.6.2 of \cite{EN} now gives us the relevance to \eqref{ACP}.

\begin{theorem}\label{ACPsoln} Let $({\ebA}, \text{\rm Dom}({\ebA}))$ be the generator of a $c_0$-semigroup
$(\emph{\texttt{V}}_t, t\geq 0)$. If $g\in {\rm Dom}({\ebA})$, then  $u_t: = \emph{\texttt{V}}_t[g]$ is a representation of the unique classical solution of \eqref{ACP}.
\end{theorem}

The reader may well have wondered where the second boundary condition in \eqref{MBC} has gone in the above formulation. This is a matter of interpretation of  $({\bT}, {\rm Dom}({\bT}))$, and hence the generator $({\bA}, {\rm Dom}({\bA}))$, as we now discuss. 
\smallskip

We are interested in the advection semigroup with exponential killing and killing on the boundary of $D$, 
\begin{equation}
\texttt{U}_t[g](i, r,\upsilon) = g(i,  r+\upsilon t, \upsilon)\mathbf{1}_{(t<\kappa^D_{r,\upsilon})}, \qquad  i = 1,\cdots, \ell\text{ and }t\geq 0.
\label{Udef}
\end{equation}
where
 \begin{equation}
\kappa_{r,\upsilon}^{D} := \inf\{t>0 : r+\upsilon t\not\in D\}.
\label{deterministic}
\end{equation}
%
%
  In essence, $(\texttt{U}_s,s\geq 0)$ is the semigroup of the process which moves from a point of issue $r$ in a straight line with velocity $\upsilon$ and which is killed on hitting $\partial D$. 
 To see why $\texttt{U}: = (\texttt{U}_s,s\geq 0)$ has the semigroup property, note that 
\[
\kappa_{r+\upsilon{s}, \upsilon} 
= \inf\{t>0 : r+\upsilon (t+{s}) \not\in D\} = (\kappa_{r, \upsilon}-{s})\vee0,
\]
so that $t<\kappa_{r+\upsilon{s}, \upsilon} $ if and only if  $t+{s} < \kappa^D_{r, \upsilon}$. Hence for any $g\in \prod_{i= 1}^m L_2({D}\times V)$ satisfying  the boundary conditions \eqref{MBC}, we have from the definition \eqref{Udef}, for $i =1,\cdots, \ell$, $r\in D$, $\upsilon\in V$,
\begin{align*}
{\texttt U}_{s} [{\texttt U}_t[g] ](i,r, \upsilon) &= 
{\texttt U}_t[g](i, r+\upsilon{s}, \upsilon)\mathbf{1}_{({s}< \kappa^D_{r, \upsilon} )}\notag\\
&=g(i, r+\upsilon(t+{s}), \upsilon)\mathbf{1}_{(t<\kappa^D_{r+\upsilon{s}, \upsilon})}
\mathbf{1}_{({s}< \kappa^D_{r, \upsilon} )}\notag\\
&=\texttt{U}_{t+{s}}[g](i,r, \upsilon) 
\end{align*}

It is a straightforward exercise, see e.g. Theorem 2 in Chapter XXI of \cite{DL6}, to  show that $\texttt{U}$ is a $c_0$-semigroup with generator 
$
{\bT}.
$
Its domain satisfies 
\begin{align}
{\rm Dom}(\bT) &= \prod_{i = 1}^\ell {\rm Dom}(\bT_i)\times \prod_{i =\ell+1}^m L_2(D\times V), \text{ where}\notag\\
{\rm  Dom}({\bT}_i)
=&\Bigg\{g\in L_2({D}\times V) : 
\upsilon\cdot\nabla g \in  L_2({D}\times V)\text{ and }g|_{\partial D^+} =0
\Bigg\}.
\label{domT}
\end{align}
Here, by $\upsilon\cdot\nabla g\in L_2({D}\times V)$ we mean that $\upsilon\cdot\nabla g$ exists in the distributional sense  and is integrable in the space $  L_2({D}\times V)$. 
%
\smallskip

 The domain of ${\bA}$ can be no larger than Dom$({\bT})$. 
It turns out however that Dom$({\bA})=$Dom$({\bT})$. To see why, we need only consider that the linear operators of the form 
\[
\texttt{K}_i f(i,r,\upsilon) :=\alpha^i(r,\upsilon)\sum_{j=1}^m\int_{V}  f(j, r, \upsilon') \pi^{i,j}(r,\upsilon,\upsilon')\d \upsilon, 
\]
are continuous mappings from $\prod_{i= 1}^m L_2({D}\times V)$ into itself, where $\alpha$ and $\pi^{i,j}$ are non-negative, measurable and uniformly bounded. The proof is a straightforward exercise which uses the Cauchy-Schwarz inequality; see for example Lemma XXI.1 of \cite{DL6}. It follows that Dom$({\bS})$  and Dom$({\bF})$ are both equal to $\prod_{i= 1}^m L_2({D}\times V)$ and, hence,  Dom$({\bA})$ and Dom$({\bT})$ agree.

\smallskip

Note there is no particular necessity to put solutions in an $L_2$ space, one might equally work with the space  $\prod_{i = 1}^m L_p (D\times V)$, for $p\in(1,\infty)$. As the reader might suspect, solutions of the backwards equation in an $L_p$ space comes hand in hand with a similarly formulated solution to the forward equation in the conjugate space $\prod_{i = 1}^m L_q(D\times V)$, where $q^{-1}+ p^{-1} =1$.
See for example Chapter XXI of \cite{DL6} or \cite{M-K}. The reader will note the exclusion of the $L_1$ and $L_\infty$ conjugacy. The reason for the exclusion boils down to the cumbersome nature of the advection operator ${\bT} = \upsilon\cdot\nabla$. Quite simply it is not possible to verify the strong continuity property of the advection semigroup
\begin{equation}
\texttt{U}_t[g](i, r,\upsilon) = g(i,  r+\upsilon t, \upsilon)\mathbf{1}_{(t<\kappa^D_{r,\upsilon})}, \qquad t\geq 0.
\label{adv}
\end{equation}
where $
\kappa_{r,\upsilon}^{D} := \inf\{t>0 : r+\upsilon t\not\in D\}.
$
Hence we cannot give a meaning to $ \upsilon\cdot\nabla$ as a $c_0$-semigroup on $L_\infty(D\times V)$. This is unfortunate as the latter is the more natural setting for probabilistic interpretation of solutions to the ACP. 
Having said that, the backwards scattering and fission operators, respectively ${\bS}$ and ${\bF}$, are well defined on all $\prod_{i = 1}^mL_p(D\times V)$ spaces for $p\in[1,\infty]$.
\smallskip

One of our main results will be to establish the asymptotic \eqref{sim} but now in the current setting. Recall that we have assumed that $D\subseteq\mathbb{R}^3$ is a smooth open pathwise connected bounded domain of concern such that $\partial D$ has zero Lebesgue measure. 

\begin{theorem}\label{leadingeval}
Let $D$ be convex. We assume the following irreducibility conditions.
For each $i,j \in \{1, \dots, \ell\}$ assume that each of the cross sections $\sigma_{\emph{\texttt{f}}}^{i}(r, \upsilon)\pi_{\emph{\texttt{f}}}^{i,j}(r, \upsilon, \upsilon')$, $\sigma^i_{\emph{\texttt{f}}}(r,\upsilon)m^j(r,\upsilon)$ and $\sigma_{\emph{\texttt{s}}}^{i}(r, \upsilon)\pi_{\emph{\texttt{s}}}^{i}(r, \upsilon, \upsilon')$  are piece-wise continuous\footnote{A function is piecewise continuous if its domain can be divided into an exhaustive finite partition (e.g. polytopes) such that there is continuity in each element  of the partition. This is precisely how cross sections are stored in numerical libraries for modelling of nuclear reactor cores.}  on $\bar{D}\times V\times V$ and  there exists $k = k_{i,j}\in\{1, \dots, \ell\}$ such that
\begin{equation}
 \sigma_{\emph{\texttt{f}}}^i(r,\upsilon)\pi_{\emph{\texttt{f}}}^{i,k}(r, \upsilon, \upsilon')
>0 \text{ on } D \times V \times V
\label{irred1}
\end{equation}
and
\begin{equation}
 \sigma_{\emph{\texttt{f}}}^k(r,\upsilon)\pi_{\emph{\texttt{f}}}^{k,j}(r, \upsilon, \upsilon')>0 \text{ on } D \times V \times V.
\label{irred2}
\end{equation}
Then, 
\begin{itemize}
\item[(i)] the neutron transport operator $\ebA$ has a simple and isolated  eigenvalue $\lambda_c > -\lambda_{\ell +1}$, which is leading in the sense that $\lambda_c = \sup\{{\rm Re}(\lambda): \lambda \text{ is an eigenvalue of }\ebA\}$ and which has  corresponding non-negative right and left eigenfunctions in $ \prod_{i=1}^mL_2(D\times V)$, $\varphi$ and $\tilde\varphi$ respectively  and
\item[(ii)] there exists an $\varepsilon>0$ such that, as $t\to\infty$,
\begin{equation}
\norm{
{\rm e}^{-\lambda_c t}{\emph{\texttt{V}}}_t[f] 
-\langle f, \tilde{\varphi} \rangle \varphi }_2 = O({\rm e}^{-\varepsilon t}),
\end{equation}
for all $f\in \prod_{i=1}^mL_2(D\times V)$, where $({\emph{\texttt{V}}}_t, t\geq 0)$ is defined in \eqref{V}.
To give a precise value for $\varepsilon$, suppose we enumerate the eigenvalues of $\ebA$ in decreasing order by the set $\{\lambda^{(1)}, \cdots, \lambda^{(n)}\}$ (noting from earlier that we have at least $\lambda^{(1)}= \lambda_c$). Then   $\lambda^{(n)}> -\lambda_{\ell + 1}$ and we can take any $\varepsilon$ such that $\varepsilon<\lambda_c- (\lambda^{(2)}\vee(-\lambda_{\ell+1} ))$ where we understand $\lambda^{(2)} = -\infty$ if $n = 1$. 
\end{itemize}
\end{theorem}

\begin{remark}\rm 
It could be argued that the assumptions in the above theorem rule out the possibility that we may, for example, include alpha  or beta emissions emissions in the model for that particular conclusion. Whilst alpha and beta emissions  may scatter, they are not energetic enough to cause fission. The irreducibility conditions \eqref{irred1} and \eqref{irred2} would thus fail. On the other hand, it is also known that when such particles are energetic enough, they can draw gamma radiation or positrons out of nuclei when passing in close proximity. If the latter are sufficiently energetic, then they  can induce fission.
\end{remark}

\section{Multi-species neutron branching process}\label{MNBPsect} Heuristically speaking, \eqref{ACP} can be thought of as  being closely related to the expectation  semigroup of a Markov branching process, or {\it Multi-species nuclear branching process (MNBP) as we shall call it}, whose infinitesimal generator is ${\bT} + {\bS}+{\bF}$. 
Consider the system of typed emissions whose configurations in $D\times V$ at time $t\ge 0$ are given by $\{r_{i,j}(t), \upsilon_{i,j}(t): i = 1, \dots , N_t^j\}$, where, for each $j = 1, \dots, m$, $N_t^j$ is the number of type $j$ emissions alive at time $t$. In order to describe the system as Markovian, we will represent it by the atomic measures
\[
X_t(j, A) = \sum_{i=1}^{N_t^j}\delta_{(r_{i,j}(t), \upsilon_{i,j}(t))}(A), \quad j = 1, \dots, m,
\]
where $A$ is a Borel subset of $D \times V$ and $\delta$ is the Dirac measure defined on the same space. Then the system can be described via the $m$-tuple $X_t(\cdot) = (X_t(1,\cdot),\dots, X_t(m,\cdot))$, $t\ge 0$, which evolves as follows.

\smallskip

$\triangleright$ A emission of type $i \in \{1, \dots, \ell\}$ with configuration $(r, \upsilon)$ moves in a straight line with velocity $\upsilon$ from the point $r$ until one of the following events occur:

\begin{itemize}
\item The emission leaves the domain, at which point it is killed.

\item Independently of all other emissions, a scattering event occurs when a emission comes in close proximity to an atomic nucleus and, accordingly, makes an instantaneous change of velocity. For an emission  in the system of type $i \in \{1, \dots, \ell\}$ with initial position and velocity $(r,\upsilon)$, if we write $T^i_{\texttt{s}}$ for the random time until the next  scattering occurs, then, independently of any other physical event that may affect the emission, 
\begin{equation}
\Pr(T^i_{\texttt{s}}>t) = \exp\left\{-\int_0^t \sigma^i_{\texttt{s}}(r+\upsilon s, \upsilon)\D s \right\}.
\label{spirit1}
\end{equation}
\item When scattering of an emission of type $i \in \{1, \dots, \ell\}$ occurs at space-velocity $(r,\upsilon)$, the new velocity is selected independently with probability $\pi^i_{\texttt{s}}(r, \upsilon, \upsilon')\d\upsilon'$. 

\item  Independently of all other emissions, a fission event occurs when an emission smashes into an atomic nucleus. For an emission in the system with initial position and velocity $(r,\upsilon)$, we will write $T^i_{\texttt{f}}$ for the random time that the next fission  occurs. Then independently of any other physical event that may affect the emission, 
\begin{equation}
\Pr(T^i_{\texttt{f}}>t) = \exp\left\{-\int_0^t \sigma^i_{\texttt{f}}(r+\upsilon s, \upsilon)\D s \right\}.
\label{spirit2}
\end{equation}
\item When fission occurs, the smashing of the atomic nucleus  releases a random number of other prompt emissions of type $i  =1,\cdots, \ell$, say $N^{i,j}\geq 0$,  which are ejected from the point of impact with randomly distributed, and possibly corollated, velocities, say $\{\upsilon^{i,j}_k: k = 1, \cdots, N^{i,j}\}$. When fission occurs at location $r\in D$ from a emission with incoming velocity $\upsilon\in{V}$, the quantity $\pi^{i,j}_{\texttt{f}}(r, \upsilon, \upsilon')\d\upsilon'$  describes the average number of type $j$ prompt emissions released from nuclear fission with outgoing velocity in the infinitesimal neighbourhood of $\upsilon'$. In particular 
\[
\int_A\pi^{i,j}_{\texttt{f}}(r, \upsilon, \upsilon')\d\upsilon' = {\rm E}\left[\sum_{k =1}^{N^{i,j}}\mathbf{1}_{(\upsilon^{i,j}_k\in A)} \right], \qquad A\in \mathcal{B}(V).
\]

\item Note, the possibility that $\Pr(N^{i,j} = 0)>0$ is possible. If $i = j = 1$ then this is tantamount to neutron capture or further decomposition into subatomic particles which are not counted.

\item Further, if the initial emission is a (type 1) neutron, a fission event (occurring at rate $\sigma^1_{\texttt{f}}$) may result in the production of unstable isotopes (which later release delayed emissions). On this event, an average number, $m^j(r, \upsilon)$, of type $j \in\{\ell+1, \dots, m\}$ isotopes will be produced from a collision at position $r$ from a neutron with incoming velocity $\upsilon$. The isotopes will inherit the configuration of the incoming nucleus at the time of collision.
\end{itemize} 
\smallskip

$\triangleright$ An isotope of type $i \in \{\ell+1, \dots, m\}$ with inherited physical configuration $(r, \upsilon)$ stays in the same place for an exponentially distributed amount of time with rate $\lambda_i$. At this point, it produces a random number of type $j \in \{1, \dots, \ell\}$ prompt emissions, the average number of which, along with their corresponding velocities, are chosen according to $\pi^{i,j}_{\texttt{f}}(r, \upsilon, \upsilon')$, in a similar way to previously described. We note that although unstable isotopes stay in the same spatial position, we will still assign them a velocity as a `mark'.

\smallskip

 In all cases, it is a natural make the following physical assumption which will remain in force throughout.
 
 \begin{assumptions}\rm
 Random emissions of any type are bounded in number  by the non-random constant $n_{\texttt{max}}\geq 1$.
In particular this means that 
\[
\sup_{1\leq i\leq m, 1\leq j\leq \ell, r\in D, \upsilon\in V}\pi^{i,j}_{\texttt f}(r,\upsilon,V)\leq n_\texttt{max}\quad\text{ and }\quad\sup_{r,\in D, \upsilon\in V_1, 1\leq j\leq \ell}m^j(r,\upsilon)\leq  {n}_{\texttt{max}}.
\]
 \end{assumptions}


\bigskip

For non-negative and uniformly bounded $g: \prod_{i = 1}^m (D\times V)\mapsto [0,\infty)$,   that is $g\in \prod_{i=1}^mL^+_{\infty}(D\times V)$, define the {\it expectation semigroup}
\begin{equation}
\psi_t[g](i, r, v) := \E_{\delta_{(i,r, v)}}[\langle g, X_t \rangle],
\label{psi}
\end{equation}
where $\mathbb{P}_{\delta_{(i,r, v)}}$ is law of the process started from a single type $i$ emission with configuration $(r,\upsilon)$ with corresponding expectation operator $\E_{\delta_{(i,r, v)}}$. 

\smallskip

As we have assumed that all cross sections are uniformly bounded, ignoring spatial trajectories of neutrons (in particular those that are killed by leaving the domain $D$), it is straightforward to 
compare the growth of $(\psi_t[g], t\geq 0)$ against that of a continuous-time Galton-Watson process with growth rate 
$
\eta\{(m\times n_{\texttt{max}})-1\}
$, where $\eta = \sup_{1\leq i\leq\ell, r\in D, \upsilon \in V}\sigma^i_{\texttt{f}}(r,\upsilon) + \max_{\ell+1\leq i\leq m}\lambda_i$.

\smallskip

The rate of growth $
\eta\{(m\times n_{\texttt{max}})-1\}
$ simply assumes that each emission of type $i$ gives rise to at most $n_{\texttt{max}}$ emissions of any other type and at a rate which is uniformly bounded by a uniform upper bound of all possible rates at which fission events occur.  Note this rate takes account of the emission  count introduced into the system at a fission event and the single emission removed from the system which caused the fission event. 

\smallskip

It is also  straightforward to stochastically upper bound the process $\langle 1, X_t\rangle$, $t\geq 0$, by the aforesaid continuous-time Galton Watson process on the same probability space.  The latter process branches whenever $X$ does, topping up the number of offspring always to $n_{\texttt{max}}$, but also  it has additional independent branching events at rate $(\eta-\mathbf{1}_{(1\leq i\leq\ell)}\sigma^i_{\texttt{f}}(r,\upsilon) - \mathbf{1}_{(\ell+1\leq i\leq m)}\lambda_i)$ always producing precisely $n_{\texttt{max}}$ offspring of each of the $m$ possible emissions.

\smallskip

If we denote this Galton-Waton process by $(Z_t, t\geq 0)$, then we have both the  stochastic bound  $\langle 1, X_t \rangle\leq Z_t\leq Z_{t+s}$, for all $s,t\geq 0$ and  the upper estimate 
\begin{equation}
\sup_{1\leq i\leq m, r\in D, \upsilon\in V}\psi_t[g](i, r, \upsilon) \leq ||g||_\infty\exp(\eta((n_{\texttt{max}}\times m)-1) t),\qquad  t\geq 0.
\label{psibound}
\end{equation}
If we put $g$ in the smaller space $\prod_{i = 1}^m C^+(D\times V)$, the space of non-negative, continuous and uniformly bounded vector functions on $(D\times V)$, then we also have by a dominated convergence argument, 
$
\lim_{t\to0}\psi_t[g] =  g$ in the pointwise sense. Otherwise the latter convergence is not necessarily clear.

\smallskip

The name `expectation semigroup' is earned thanks to the behaviour of $(\psi_t, t\geq 0)$ under an application of the  Markov branching property. 
Indeed, associated to the MNBP are the probabilities $\mathbb{P}_{\mu}$ for atomic measures of the form 
\begin{equation}
\mu=\left( \sum_{i =1}^{n_1}  \delta_{(1, r_{i, 1}, \upsilon_{i, 1})}, \cdots,\sum_{i =1}^{n_m}  \delta_{(m, r_{i, m}, \upsilon_{i, m})}.
\right) =: (\mu_1,\cdots,\mu_m).
\label{atomicmeasure}
\end{equation}
The  Markov branching property dictates that, for  $g\in \prod_{i = 1}^m L_2 (D\times V)$ as before and $t\geq 0$, 
\[
\mathbb{E}_{\mu}[\langle g, X_t\rangle] =  \sum_{j =1}^{m}  \sum_{i =1}^{n_j}
\mathbb{E}_{\delta_{(j, r_{i, j}, \upsilon_{i, j})}}[\langle g, X_t\rangle] 
= \langle \mathbb{E}_{\delta_{(\cdot, \cdot, \cdot)}}[\langle g, X_t\rangle] ,  \mu \rangle
\]
Here we are abusing our earlier notation in \eqref{innerprods} and writing for finite atomic measures $\mu$ of the form \eqref{atomicmeasure}, 
\begin{equation}
\langle  g,\mu \rangle = \sum_{i = 1}^m ( g, \mu)_i,
\quad\text{ 
where 
 }\quad
(g,\mu)_i = \int_{D\times V} g(i, r, \upsilon) \mu_i(\d r, \d\upsilon) .
\label{innerprods2}
\end{equation}
Hence, by conditioning on the configuration of the system at time $t\geq0$, we have, for $s\geq 0$,
\begin{equation}
\psi_{t+s}[g](i, r, v) := 
\E_{\delta_{(i,r, v)}}
\left[
\E_{X_t}
[   \langle f, X_s \rangle  ]\right]=\E_{\delta_{(i,r, v)}}
\left[
\langle  \psi_s [g],X_t \rangle
\right]
= \psi_{t}[\psi_s[g]](i, r, v).
\label{expectationsemigroup}
\end{equation}

The expectation semigroup property of $(\psi_t, t\geq 0)$   does not  imply that it is necessarily a $c_0$-semigroup on $\prod_{i= 1}^m L_2({D}\times V)$. Recalling our earlier discussion, if we were able to work with \eqref{ACP} in the setting of a $c_0$-semigroup on $\prod_{i=1}^m L_\infty(D\times V))$, then we would be much closer to being able to match the expectation semigroup $(\psi_t, t\geq 0)$ to the solution $(u_t, t\geq 0)$. But even then, problems would occur with verifying strong continuity at the origin.

\smallskip

 Nonetheless, classical literature supports the view that it is the physical processes, i.e. in this setting the MNBP, that provides a stochastic representation of the solution to the backward MNTE. The authors are not aware of a formal proof of this fact. We will nonetheless try to address this point shortly in Section \ref{consolidate}. In the mean time, let us present an alternative `mild' form of the MNTE (also called a {\it Duhamel solution} in the PDE literature) which the semigroup $(\psi_t,t\geq 0)$ more comfortably solves. 
 \begin{lemma}\label{mild}
The expectation semigroup  $(\psi_t[g], t\geq 0)$  is the unique  solution in $\prod_{i=1}^m L^+_\infty (D\times V)$ to the mild MNTE 
\begin{equation}
u_t(i,r,\upsilon)  =\emph{\texttt U}_t[g](i,r,\upsilon) + \int_0^t \emph{\texttt U}_s[({\ebS+\ebF})u_{t-s}](i,r,\upsilon)\d s,
\label{mildeq}
\end{equation}
for $t\geq 0$, $1\leq i\leq m$, $r\in D, \upsilon\in V$ and  $g\in \prod_{i = 1}^m L^+_\infty(D\times V)$. 
\end{lemma}
Before proceeding to the proof, let us remark that, in the statement of the theorem, we are not working with $(\texttt{U}_t, t\geq 0)$ as a $c_0$-semigroup on $\prod_{i = 1}^m L_\infty(D\times V)$, but a pointwise shift operator. The reader will recall from the discussion preceding \eqref{adv} that $(\texttt{U}_t, t\geq 0)$ cannot be defined as such for $\prod_{i = 1}^m L_\infty(D\times V)$.

%
%
%

%
%
%

\begin{proof}[Proof of Lemma \ref{mild}] First suppose we start with an emission of type $i$.
By splitting the expectation in the definition of $\psi_t[g]$ at the first scattering or fission event, and remembering that the time $\kappa_{r,\upsilon}^{D}$ defined in \eqref{deterministic} is deterministic, we have for $r\in D$ and $\upsilon\in V$,
\begin{align*}
&\psi_t[g](i,r,\upsilon) \\
& ={\rm e}^{-\int_0^t\sigma^i(r+\upsilon s,\upsilon)\d s}g(i,r+\upsilon t,\upsilon)\mathbf{1}_{(t<\kappa_{r,\upsilon}^{D} )}\\
&\hspace{1cm}+\int_0^{t\wedge\kappa_{r,\upsilon}^{D} } \sigma^i(r+\upsilon s,\upsilon){\rm e}^{-\int_0^s\sigma^i(r+\upsilon u,\upsilon)\d u}\\
&\hspace{2cm}\Bigg\{
\frac{\sigma^i_{\texttt{s}}(r+\upsilon s,\upsilon)}{\sigma^i(r+\upsilon s,\upsilon)}\int_{V}\psi_{t-s}[g](i, r+\upsilon s, \upsilon') \pi_{\texttt s}^i(r+\upsilon s, \upsilon, \upsilon')\d\upsilon'\\
&\hspace{3cm}+
\frac{\sigma^i_{\texttt{f}}(r+\upsilon s,\upsilon)}{\sigma^i(r+\upsilon s,\upsilon)} \Bigg(\sum_{j =1}^m\int_{V}\psi_{t-s}[g](j, r+\upsilon s, \upsilon') \pi^{i,j}_{\texttt f}(r+\upsilon s, \upsilon, \upsilon')\d\upsilon' \\
&  \hspace{6cm}+\mathbf{1}_{(i=1)}\sum_{j = \ell+1}^m m^j(r+\upsilon s, \upsilon)\psi[g](j, r+\upsilon s, \upsilon,  t-s)\Bigg)\Bigg\}\d s\\
&={\rm e}^{-\int_0^t\sigma^i(r+\upsilon s,\upsilon)\d s}g(i, r+\upsilon t,\upsilon)\mathbf{1}_{(t<\kappa_{r,\upsilon}^{D} )}\\
&\hspace{2cm}+\int_0^{t\wedge\kappa_{r,\upsilon}^{D} } {\rm e}^{-\int_0^s\sigma^i(r+\upsilon u,\upsilon)\d u}
({\bS}_i+{\bF}_i+\sigma^i)\psi_{t-s}[g](i, r+\upsilon s, \upsilon)\d s, \qquad t\geq 0.
\end{align*}
 Now appealing to an analogue of  Lemma 1.2, Chapter 4 in \cite{Dynkin2} (see also the Appendix of \cite{KP}), we can transfer the exponential integrals in each of the terms on the right-hand side above to a potential term in the integral so that we end with 
 \begin{equation}
 \psi_t[g](r,\upsilon)= g(i, r+\upsilon t,\upsilon)\mathbf{1}_{(t<\kappa_{r,\upsilon}^{D} )}
+\int_0^{t\wedge\kappa_{r,\upsilon}^{D} } ({\bS}_i+{\bF}_i)\psi_{t-s}[g](i, r+\upsilon s, \upsilon')\d s, \qquad t\geq 0,
\label{<l}
 \end{equation}
which agrees with \eqref{mildeq}, for $1\leq i\leq \ell.$

\smallskip

Following a similar approach, for $\ell+1\leq i\leq m$, $r\in D$, $\upsilon,\in V$, we also get
\begin{align}
\psi_t[g](i,r,\upsilon) &=g(i, r+\upsilon t,\upsilon)\mathbf{1}_{(t<\kappa_{r,\upsilon}^{D} )}
 - \lambda_i \int_0^{t\wedge\kappa_{r,\upsilon}^{D} }\psi_{t-s}[g](i, r+\upsilon s, \upsilon)\D s\notag\\
&\qquad+\int_0^{t\wedge\kappa_{r,\upsilon}^{D} }\lambda_i\left\{\sum_{j=1}^\ell\int_{V}\psi_{t-s}[g](1, r+\upsilon s, \upsilon')\pi_{\texttt{f}}^{i,j}(r+\upsilon s, \upsilon, \upsilon')\D \upsilon' \right\}\D s\notag\\
&=g(i, r+\upsilon t,\upsilon)\mathbf{1}_{(t<\kappa_{r,\upsilon}^{D} )}
+\int_0^{t\wedge\kappa_{r,\upsilon}^{D} } ({\bS}_i+{\bF}_i)\psi_{t-s}[g](i, r+\upsilon s, \upsilon')\d s,\qquad t\geq 0,
\label{>l}
\end{align}
noting in particular that, for $\ell+1\leq i\leq m$, ${\bS}_i \equiv 0$. Now putting \eqref{<l} and \eqref{>l} together we obtain \eqref{mildeq}.
\smallskip

For uniqueness, suppose that $(\psi^{(i)}_t, t\geq 0)$, $i= 1,2$ are two bounded solutions to \eqref{mildeq}. Define $\chi_t[g]: =|\psi^{(1)}_t[g] - \psi^{(2)}_t[g]|$ and note that, for $i=1\cdots,m$,
\begin{align}
\chi_t[g](i, r,\upsilon)& \leq\int_0^{t\wedge \kappa_{r,\upsilon}^{D} } |({\bS}+{\bF}) \psi^{(1)}_{t-s}[g](i,r+\upsilon s, \upsilon) - ({\bS}+{\bF}) \psi^{(2)}_{t-s}[g](i, r+\upsilon s, \upsilon)|   \d s  \notag\\
&\leq \int_0^{t\wedge \kappa_{r,\upsilon}^{D} } ({\bS}+{\bF})|\psi^{(1)}_{t-s}[g](i,r+\upsilon s, \upsilon) - \psi^{(2)}_{t-s}[g](i, r+\upsilon s, \upsilon)|   \d s  \notag\\
&\leq \int_0^{t\wedge \kappa_{r,\upsilon}^{D} } ({\bS}+{\bF})\chi_{t-s}[g](i,r+\upsilon s, \upsilon)  \d s  \notag\\
&\leq C_1 \int_0^{t\wedge \kappa_{r,\upsilon}^{D} } \sum_{j=1}^m\int_{V} \chi_{t-s}[g](j, r+\upsilon s, \upsilon')\d\upsilon'   \d s
+C_2\int_0^{t\wedge \kappa_{r,\upsilon}^{D} }   \chi_{t-s}[g](i, r+\upsilon s, \upsilon)\d s
\label{plug1}
\end{align}
for some constants $C_1, C_2\in(0,\infty)$, where the final inequality follows on account of all cross sections being uniformly bounded. 
 Now  define $ \bar\chi_t[g]=\sup_{1\leq i\leq m, r\in {D}, \upsilon\in{V}}\chi_t[g](i,r, \upsilon)$, $t\geq 0$. From \eqref{plug1} we have  that
\begin{align}
\bar\chi_t[g]&\leq \left(C_1\sum_{j = 1}^m \texttt{Vol}(V) +C_2\right)\int_0^{t }\bar\chi_{t-s}[g] \d s.
\end{align}
Reversing the order of integration on the right-hand side above and then applying Gr\"onwall's Lemma allows us to conclude that $\chi_t[g]\equiv0 $, which shows uniqueness. 
\end{proof}


\section{Multi-species neutron random walk and the Many-to-one Lemma}\label{RW}
A second  probabilistic perspective for analysing the MNTE is possible, seems rarely to have been discussed in existing literature, if at all. This consists of collapsing the sum of the operators ${\bT} + {\bS}+{\bF}$ to take the form ${\bL}+\diagonal ({\beta})$ for an appropriate choice of $\beta$, where $\bL$ is the operator which is similar in structure to ${\bT} + {\bS}$. In essence, this transformation, which we will describe more rigorously in a moment, heuristically postulates that the operator ${\bT} + {\bS}+{\bF}$ can be reinterpreted via a Feynman-Kac formula as the infinitesimal generator of a single emission which undergoes linear transport and scattering and which also accumulates potential $\beta$.

\smallskip

To describe this more precisely, we need to introduce the notion of a 
 {\it multi-species
neutron random walk} (MNRW). In the current setting this means a continuous-time typed random walk by $ (J_t, R_t, \Upsilon_t)$, $t\geq 0$, on $\{1,\cdots,m\}\times (D\times V)$ with additional cemetery state $\{\dagger\}$ when it exits the physical domain $D$ or an emission otherwise disappears from the system. 
The MNRW is described by two fundamental quantities (which are functions of the current particle type, spatial position and velocity). First,  a scattering rate $\alpha^i(r,\upsilon)$, $i\in\{1,\cdots,m\}, r\in D, \upsilon,\upsilon'\in V$, such that $\alpha^i(r,\upsilon) = \lambda_i$, for $i \in\{\ell+1,\cdots,m\}$. Second, a scattering probability kernel $\pi^{i,j}(r,\upsilon, \upsilon')$, $i,j\in \{1,\cdots, m\}, r\in D, \upsilon,\upsilon'\in V$.
In the spirit of the description of the MNBP, the MNRW is described as follows.

\smallskip

$\triangleright$ When the MNRW is of type $i \in \{1, \dots, \ell\}$ with configuration $(r, \upsilon)$, it moves in a straight line with velocity $\upsilon$ from the point $r$ until one of the following events occur:

\begin{itemize}
\item When the MNRW position moves out of $D$ or e.g. it decomposes into an emission type that is not counted, or is captured in a nucleus, it is instantaneously killed.

\item A scattering event occurs  and, accordingly, the MNRW keeps the same emission type but makes an instantaneous change of velocity. If we write $T^i_{\texttt{s}}$ for the random time until the next  scattering occurs, then, 
\begin{equation}
\Pr(T^i_{\texttt{s}}>t) = \exp\left\{-\int_0^t \alpha^i(r+\upsilon s, \upsilon)\D s \right\}.
\label{alphascatter}
\end{equation}
\item When scattering of an emission of type $i \in \{1, \dots, \ell\}$ occurs at space-velocity $(r,\upsilon)$, the new velocity is selected independently with probability $\pi^i(r, \upsilon, \upsilon')\d\upsilon'$.
\end{itemize}

$\triangleright$  Otherwise, if $\ell+1\leq i\leq m$, then the emission remains motionless, i.e. the random walk is dormant, holding its initial position $r$, but retaining the velocity $\upsilon$ as a mark.
After an independent and exponentially distributed random time with rate $\lambda_i$, the particle transfers it type $j\in\{1,\cdots,\ell\}$ and acquires a new velocity $\upsilon'$ with probability density $\pi^i(r,\upsilon,\upsilon')$.

\smallskip

We can associate to the MNRW the infinitesimal generator
\begin{align}
 {\bL}_if(r,v) &: =  \mathbf{1}_{(i\le \ell)}\upsilon \cdot \nabla f(i, r,\upsilon)\mathbf{1}_{(r\in D)}\notag\\
 &+ \alpha^i(r, \upsilon)\sum_{j=1}^m\int_{V}[f(j, r, \upsilon') - f(i, r, \upsilon)]\pi^{i,j}(r, \upsilon, \upsilon')\D \upsilon'.
 \label{L}
\end{align}
for $f\in \text{Dom}({\bL}) = \text{Dom}({\bT})$. We thus refer to the process as an $\bL$-MNRW.

\smallskip

With the notion of the MNRW in hand, let us consider the following algebraic manipulations.
 For $i\in\{1, \dots, \ell\}$, $j\in\{1, \dots, m\}$, $(r,\upsilon) \in D\times V$ and $\upsilon' \in V$, define
\begin{align}
\alpha^i(r,\upsilon) &= \mathbf{1}_{(1\leq i\leq \ell)}\sigma_{\texttt{s}}^i(r,\upsilon) \notag\\
&\hspace{1cm}+\mathbf{1}_{(1\leq i\leq \ell)} \sigma_{\texttt{f}}^i(r,\upsilon)\Bigg(\sum_{j=1}^\ell\int_V\pi_{\texttt{f}}^{i,j}(r, \upsilon,\upsilon')\d\upsilon'+\mathbf{1}_{(i=1)}\sum_{j =\ell+1}^m m^j(r,\upsilon) \Bigg)\notag\\
&\hspace{5cm}+\mathbf{1}_{(\ell+1\leq i\leq m)}\lambda_i\sum_{j=1}^\ell \int_V\pi_{\texttt{f}}^{i,j}(r, \upsilon, \upsilon')\d\upsilon',
\label{alpha}
\end{align}
\begin{align}
\pi^{i,j}(r, \upsilon, \upsilon') &= (\alpha^i(r,\upsilon))^{-1}\Bigg[\sigma_{\texttt{s}}^i(r,\upsilon)\pi_{\texttt{s}}^i(r, \upsilon, \upsilon')\mathbf{1}_{(1\leq i=j\leq \ell)} \notag\\
&\hspace{3cm}+ \sigma_{\texttt{f}}^i(r,\upsilon)\left(\pi_{\texttt{f}}^{i,j}(r, \upsilon, \upsilon')\mathbf{1}_{(1\le i, j\le \ell)} + m^j(r,\upsilon)\mathbf{1}_{(i=1, j>\ell))} \right)\notag\\
&\hspace{7.5cm}
+\lambda_i\pi_{\texttt{f}}^{i,j}(r, \upsilon, \upsilon')\mathbf{1}_{(\ell+1 \le i \le m, \, j\le \ell)}\bigg],\label{pi}\\
\beta^i(r,\upsilon) &= \alpha^i(r,\upsilon) - \mathbf{1}_{(1\leq i\leq \ell)} \sigma^i_{\texttt{s}}(r,\upsilon) 
- \mathbf{1}_{(\ell+1\leq i\leq m)}\lambda_i- \mathbf{1}_{(1\leq i\leq \ell)} \sigma^i_{\texttt{f}}(r,\upsilon) .\label{beta}
\end{align}
Note, in particular, that for each fixed $1\leq i\leq m$, $r\in D$ and $\upsilon\in V$, $\pi^{i,j}(r,\upsilon,\upsilon')$ is a probability distribution on $\{1,\cdots, m\}\times V$ in the sense that $\sum_{j=1}^m \int_{V} \pi^{i,j}(r,\upsilon,\upsilon')\d \upsilon' = 1$.
Note also that the assumption $\sum_{j= 1}^\ell \int_V\pi_{\texttt{f}}^{i,j}(r, \upsilon,\upsilon')\d\upsilon'\geq0$ ensures that $\beta^i\geq0$, for $1\leq i\leq m$.

\smallskip

With simple algebra, we may now identify 
\begin{equation}
({\bT}+ {\bS}+ {\bF})f ( r, \upsilon)
= {\bL} f ( r, \upsilon)+ \texttt{diag}(\beta) f ( r, \upsilon)  
\label{compact}
\end{equation}
where, for $f\in  $ Dom$({\bA}) $ (for which it was remarked earlier that it is  equal to $\text{Dom}({\bT})$),  and 
${\bL}$ is given by \eqref{L}.

\smallskip

Heuristically speaking, we have algebraically gathered all of the operators into the infinitesimal generator of an $\bL$-MNRW  and local potential $\beta$. This has the attraction of leading us the aforementioned single emission representation of the solution to the MNTE using a single-emission Feynman-Kac representation.  
%
%
%
%
%
 Said another way, this means that one would expect that, in the appropriate sense, the solution to the NTE to be represented in the form 
\begin{equation}
\phi_t[g](i,r,\upsilon)  = \mathbf{E}_{(i, r,\upsilon)}\left[{\rm e}^{\int_0^t\beta^{J_s}(R_s, \Upsilon_s)\D s}g(J_t, R_t, \Upsilon_t) \mathbf{1}_{(t < \tau_D)}\right], 
\label{phi}
\end{equation}
for  $t\geq 0$, $1\leq i\leq m$, $r\in D, \upsilon\in V$. Here
 ${\bf P}_{(i,r, v)}$ for the law of the $\bL$-MNRW  starting from a single emission with configuration $(i, r, \upsilon)$, and $\er_{(i,r, v)}$ for the corresponding expectation operator.

\smallskip

Appealing to the Markov property for $(J, R, \Upsilon)$, it is not difficult to show that a semigroup property similar to \eqref{expectationsemigroup} holds. That is to say, for $s,t\geq 0$, $1\leq i\leq m$, $r\in D, \upsilon\in V$
\[
\phi_{s+t}[g](i,r,\upsilon) = \phi_{s}[\phi_t[g]](i,r,\upsilon).
\] 
Similarly to the case of $(\psi_t[g], t\geq 0)$, if we put $g$ in the smaller space $\prod_{i = 1}^m C^+(D\times V))$ then we also have $\lim_{t\to0}\phi_t[g] = g$ in the pointwise sense, but otherwise strong continuity at $t = 0$ is unclear. 
Note also  that, since all cross sections are uniformly bounded, 
then so is $\beta$ (in all of its variables) by a constant, say $\bar\beta$. Hence, for  $g\in \prod_{i = 1}^m L_\infty(D\times V)$, the $\phi_t[g]\leq \norm{g}_\infty\exp(\bar\beta t)$, $t\geq 0$.
As with the case of $(\psi_t[g], t\geq 0)$, the notion that $(\phi_t[g], t\geq 0)$, solves \eqref{ACP} is not a straightforward claim. Nonetheless, as one might expect, these two expectation semigroups are equal and, we can see this by relating back to \eqref{mildeq}.
\smallskip

Indeed, by conditioning the expectation in the definition of $\phi_t[g]$ on the first scattering event, and then appealing to the Lemma 1.2, Chapter 4 in \cite{Dynkin2} in a similar manner to what was done in the proof of Lemma \ref{mild}, one easily deduces the below result. In the the spatial branching process literature, this would be called a `{\it many-to-one}' lemma.

\begin{lemma}
For $g\in \prod_{i = 1}^m L^+_\infty(D\times V)$, the two expectation semigroups $(\phi_t[g], t\geq 0)$ and $(\psi_{t}[g], t\geq0)$ agree.
\end{lemma}

\section{Consolidating the ACP with the expectation semigroup}\label{consolidate}
We want to understand how the $\prod_{i= 1}^m L_2({D}\times V)$ semigroup $(\texttt{V}_t, t\geq 0)$ that represents the unique solution to the Abstract Cauchy Problem \eqref{ACP} relates to the expectation semigroups $(\psi_t, t\geq 0)$ and $(\phi_t, t\geq 0)$ that offer two different stochastic representations to the mild equation \eqref{mildeq}.

\smallskip

We start by noting that if $g\in \prod_{i = 1}^m L^+_\infty(D\times V)$, then, on account of the fact that ${\rm Vol}(\prod_{i =1}^m (D\times V)) =( \int_{D\times V}\d r\d\upsilon)^m<\infty$, we also have $g\in \prod_{i= 1}^m L_2({D}\times V)$. Since it is unclear whether $(\psi_t[g], t\geq 0)$ is well defined for all $g\in \prod_{i= 1}^m L_2({D}\times V)$, it makes  makes sense  to consider the comparison with $(\texttt{V}_t[g], t\geq 0)$ (defined in \eqref{V}) for the more restrictive choice $g\in \prod_{i = 1}^m L_\infty(D\times V)$. The natural setting in which to make the comparison is in the space $\prod_{i= 1}^m L_2({D}\times V)$ as, by \eqref{psibound}, $\norm{\psi_t[g]}_\infty<\infty$ and the latter implies $\norm{\psi_t[g]}_2<\infty$, again  thanks to the fact that ${\rm Vol}(\prod_{i =1}^m (D\times V))<\infty.$

\begin{theorem}\label{vtpsit}
If $g\in \prod_{i = 1}^m L^+_\infty(D\times V)$ then, for $t\geq 0$,  $\emph{\texttt{V}}_t[g] = \psi_t[g]$ on $\prod_{i= 1}^m L_2({D}\times V)$, i.e.
$\norm{\emph{\texttt{V}}_t[g] - \psi_t[g]}_2 = 0$. 
\end{theorem}
Before moving to its proof, the reader should take care to note that this does not imply that $(\texttt{V}_t, t\geq 0)$ and $(\psi_t, t\geq 0)$ agree as $c_0$-semigroups on $\prod_{i= 1}^m L_2({D}\times V)$. In particular,  the comparison between the two semigroup operators is only made on $\prod_{i = 1}^m L_2(D\times V)$, and $(\psi_t, t\geq 0)$ was not   (and in fact cannot be) shown to demonstrate the strong continuity property on $\prod_{i= 1}^m L_2({D}\times V)$.

\begin{remark}\rm If we consider 
Theorem \ref{vtpsit} in light of Theorem \ref{leadingeval}, noting that $(\psi_t[g],t\geq0)$ is a uniformly bounded sequence,  it is tempting to want to say that the leading eigenfunction $\varphi$ belongs to $\prod_{i = 1}^m L_\infty(D\times V)$. This is not the case necessarily and remains to be proved. In the setting of a single type of emission, this will be demonstrated in the forthcoming paper \cite{SNTE}.
\end{remark}
\begin{proof}[Proof of Theorem \ref{vtpsit}]
Consider the adjusted ACP with inhomogeneity given by 
\begin{equation}
\left\{
\begin{array}{rl}
\dfrac{\partial u_t}{\partial t} &= {\bT} u_t+({\bS}+{\bF})\texttt{V}_{t}[g]
\\
u_0& = g 
\end{array}
\right.
\label{adjACP}
\end{equation}
By taking the difference of two solutions and invoking the uniqueness of the ACP in  $\prod_{i= 1}^m L_2({D}\times V)$ with initial data $g = 0$, we note that the solution to \eqref{adjACP}  is  unique in $\prod_{i= 1}^m L_2({D}\times V)$. However, on the one hand, it is straightforward to verify that 
\[
u_t= {\rm e}^{t{\bT}}g + \int_0^t{\rm e}^{(t-s){\bT}} ({\bS}+{\bF})\texttt{V}_{s}[g] \d s, \qquad t\geq 0,
\]
solves \eqref{adjACP}.
On the other hand, taking account of the fact that $({\texttt V}_t[g],t\geq 0)$ solves \eqref{ACP}, it is also the case that 
\[
u_t = {\texttt V}_t[g], \qquad t\geq 0,
\]
solves \eqref{adjACP}. Uniqueness thus tells us that on $\prod_{i= 1}^m L_2({D}\times V)$, 
\[
{\texttt V}_t[g] = {\texttt U}_t[g]+ \int_0^t{\rm e}^{(t-s){\bT}} ({\bS}+{\bF})\texttt{V}_{s}[g] \d s  = 
{\texttt U}_t[g]+ \int_0^t\texttt{U}_s[ ({\bS}+{\bF})\texttt{V}_{t-s}[g]] \d s, \qquad t\geq 0,
\]
where in the second equality we have reversed the direction of integration. 
In conclusion, where as $(\psi_t[g], t\geq 0)$ solves \eqref{mildeq} in the pointwise sense, $({\texttt V}_t[g], t\geq 0)$ solves it in the $\prod_{i= 1}^m L_2({D}\times V)$ sense.

\smallskip

On the other hand, we know that $(\psi_t[g], t\geq 0)$ is valued in $\prod_{i= 1}^m L_2({D}\times V)$, hence we can consider,
\[
\norm{\psi_t[g] - {\texttt V}_t[g]}_2
= \norm{\int_0^t \texttt{U}_s[ ({\bS}+{\bF})\{\psi_{t-s}[g]-\texttt{V}_{t-s}[g]\}]\d s}_2, \qquad t\geq 0.
\]
To this end, let us note that, for $T>0$, and $w_t\in \prod_{i= 1}^m L_2({D}\times V)$, $t\leq T$, we have 
\begin{align}
\norm{\int_0^t w_s\d s}_2^2 &= \int_{D\times V} \left(t\int_0^tw_s (r,\upsilon)\frac{\d s}{t}\right)^2\d r\d\upsilon \notag\\
&\leq  \int_{D\times V} t^2\left(\int_0^tw_s (r,\upsilon)^2\frac{\d s}{t}\right)\d r\d\upsilon\notag\\
&\leq T\int_0^t \norm{w_s}^2_2 \d s, \qquad t\leq T,
\label{normint}
\end{align}
where in the first inequality we have used Jensen's inequality and Cauchy-Schwarz  in the second.
Moreover, for $f\in \prod_{i= 1}^m L_2({D}\times V)$,
\begin{align}
\norm{\texttt{U}_s[f]}^2_2 &= \sum_{i =1}^m\int_{D\times V} \mathbf{1}_{(s<\kappa_{r,\upsilon}^D)}f(i, r+\upsilon s, \upsilon)^2\d r\d \upsilon\notag\\
&\leq \sum_{i =1}^m\int_{D\times V}  f(i, r', \upsilon)^2\d r '\d \upsilon \notag\\
&=\norm{f}_2^2
\label{subdomain}
\end{align}
where the inequality follows as a consequence that, for each $\upsilon$, the integral $\int_D \mathbf{1}_{(s<\kappa_{r,\upsilon}^D)}v(i, r+\upsilon s, \upsilon)^2\d r$ integrates over a subdomain of $D$.
Also, we have for the operator ${\bS}$ (and similarly for ${\bF}$), for $f\in \prod_{i= 1}^m L_2({D}\times V)$,
\begin{align}
\norm{({\bS}+\diagonal(\sigma_{\texttt{s}})) f}_2 &= \left(\sum_{i = 1}^m\int_{D\times V}
\left(\int_{V} f(i,r,\upsilon')\sigma_{\texttt{s}}(r,\upsilon)\pi^i_{\texttt{s}}(r,\upsilon,\upsilon')\d\upsilon'\right)^2
\d r\d\upsilon \right)^{1/2}\notag\\
&\leq C\left(\sum_{i = 1}^m\int_{D\times V}
\left(\int_{V} f(i,r,\upsilon')\times 1\,\d\upsilon'\right)^2
\d r\d\upsilon\right)^{1/2}\notag\\
&\leq C
\left(\sum_{i = 1}^m {\rm Vol}(V)\int_{D\times V} \int_{V}f(i, r, \upsilon')^2\d\upsilon'\d r\right)^{1/2}\notag\\
&\leq C \max_{1\leq i\leq m}{\rm Vol}(V) \norm{f}_2,
\label{Sftof}
\end{align}
where the constant $C$ appears by upper estimating the uniformly bounded cross sections  and in the second inequality we have used Cauchy-Schwarz.
\smallskip

It thus follows from \eqref{normint}, \eqref{subdomain} and \eqref{Sftof} that, for $t\leq T$,  writing $\omega_t = \psi_t[g]-\texttt{V}_{t}[g]$, $t\geq 0$,
\begin{align}
\norm{\omega_t}_2^2&= \left\|\int_0^t \texttt{U}_s[ ({\bS}+{\bF})\omega_{t-s}]\d s\right\|^2_2\notag\\
&\leq   T\int_0^t \norm{\texttt{U}_s[ ({\bS}+{\bF})\omega_{t-s}]}_2^2\d s\notag\\
&\leq  T\int_0^t\norm{({\bS}+{\bF} )\omega_{t-s}}_2^2\d s\notag\\
&=  T\int_0^t\norm{({\bS}+{\bF}+\diagonal(\sigma)-\diagonal(\sigma) )\omega_{s}}_2^2\d s\notag\\
&\leq  T\int_0^t
\left(\norm{({\bS}+ \diagonal(\sigma_{\texttt s}))\omega_s}_2 +\norm{({\bF}+ \diagonal(\sigma_{\texttt f}))\omega_s}_2 +\norm{\diagonal(\sigma)\omega_{s}}_2\right)^2
\d s\notag\\
&\leq  C' \int_0^t\norm{\omega_s}_2^2\d s,\qquad t\leq T,
\label{omega}
\end{align}
where the constant $C'$ comes from the fact that $\sigma$ is uniformly bounded.
%
The final inequality in \eqref{omega} together with Gr\"onwall's Lemma now tells us that $\norm{\omega_t}_2 = 0$, for all $t\leq T$. Since $T$ is chosen arbitrarily, it follows that $(\psi_t[g],t\geq 0)$ and $(\texttt{V}_t[g], t\geq 0)$ are indistinguishable in $\prod_{i= 1}^m L_2({D}\times V)$.
\end{proof}

The conclusion of this section is that it is not unreasonable to now understand the expectation semigroups $(\psi_t[g],t\geq 0)$ and $(\phi_t[g],t\geq 0)$ for non-negative, bounded and measurable $g$ on $D\times V$ as the `solution' to the MNTE in place of $(\texttt{V}_t[g], t\geq 0)$ for the same class of $g$. Indeed, the two agree both in $\prod_{i= 1}^m L_2({D}\times V)$ and hence  $(\d r\times \d \upsilon)$-Lebesgue almost everywhere.

\smallskip

The reader will also note that from the perspective of Monte Carlo simulation, the expectation semigroup $ (\phi_t[g],t\geq 0)$ carries the potential to be exploited in a way that $ (\psi_t[g],t\geq 0)$ cannot. More precisely, where branching trees are difficult to simulate and are not convenient for  Monte Carlo computational parallelisation, random walks are. This simple idea is explored in greater detail in the accompanying paper to this one \cite{MCNTE}.

\section{Asymptotic behaviour of the MNTE: Proof of Theorem \ref{leadingeval}}\label{proof}
In this  section we return to the fundamental notion that the solution to the MNTE in the form \eqref{ACP} is described by its leading asymptotics for large times. That is to say, we give the proof of Theorem \ref{leadingeval}.  Our proof follows closely ideas found in Chapters 4 and 5 of \cite{M-K}.
\smallskip

Recall that the quantities $\alpha^i$, $\pi^{i,j}$, $\beta^i$, $i,j  =1,\cdots,m$ were defined in \eqref{alpha}, \eqref{pi} and \eqref{beta} respectively. They were arranged into the operator $\bA= \bT+\bS+\bF$, such that Dom$(\bA)=$  Dom$({\bT})$, described in \eqref{domT}.
 \smallskip

For $j = 1, \dots,m$, let us introduce the operators $\K_{i,j}$ on  $L_2(D \times V) $ by
\[
\K_{i,j}f(r,\upsilon) = \alpha^i(r,\upsilon)\int_{V}f(r, \upsilon')\pi^{i,j}(r, \upsilon, \upsilon')\D \upsilon' .
\]
These are  integral operators, which take the form
\[
\K_{i,j}f(r,\upsilon)  = \int_{V} f(r,\upsilon') \texttt{k}_{i,j}(r,\upsilon,\upsilon')\d \upsilon'
\]
on $D\times V\times V$,
where
\begin{equation}
\texttt{k}_{i,j}(r,\upsilon,\upsilon') = \sigma^i_{\texttt{s}}\pi^i_{\texttt{s}}(r,\upsilon,\upsilon')+  \sigma^i_{\texttt{f}}\pi^{i,j}_{\texttt{f}}(r,\upsilon,\upsilon').
\label{k}
\end{equation}
A similar computation to \eqref{Sftof}  also shows that $\K_{i,j}g \in L_2(D \times V)$ when $g\in L_2(D \times V)$.
Then from \eqref{promptNTE} and \eqref{pi}, taking care to note the use of the indicators for the inclusion of terms for different indices,  we can write, for $1\leq i\leq \ell$, for $g\in {\rm Dom}(\bA)$,
 \begin{align}
\bA_i g(i,r,\upsilon)&= \bT_ig(i,r,\upsilon) -\sigma^i (r,\upsilon)g(i,r,\upsilon)\notag\\
& \hspace{1cm}+\sum_{j = 1}^\ell\K_{i,j}g(j, r,\upsilon)
+ \mathbf{1}_{(i=1)}\sigma^1(r, \upsilon)\sum_{j =\ell +1}^m m^j(r,\upsilon)g(j,r,\upsilon)
 \label{genND2}
\end{align}
Moreover, for $\ell+1\leq i\leq m$, 
 \begin{align}
\bA_i g(i,r,\upsilon)&= -\lambda_ig(i, r,\upsilon) + \sum_{j=1}^\ell \K_{i,j}g(j, r,\upsilon)
 \label{genD}
\end{align}

\smallskip

With this notation, write
\begin{align*}
\T&=\diagonal({\bT_1-\sigma^1,\cdots,\bT_\ell}-\sigma^\ell)  ,\\
\Lambda &= \diagonal(\lambda_{\ell +1}, \dots, \lambda_m),\\
\K^\circ &= (\K_{i,j}), \quad \text{ for } i,j = 1, \dots, \ell,\\
\M &= (\M_{i,j}), \quad \text{ where } \M_{i,j} = \sigma^1(r,\upsilon)m^j(r,\upsilon)\mathbf{1}_{(i=1)}, \text{ for } i = 1,\dots, \ell, j = \ell+1,\dots, m,\\
\K_\circ &= (\K_{i,j}), \quad \text{ for }  i = \ell+1, \dots, m, j = 1, \dots, \ell. 
\end{align*}
Then the abstract Cauchy problem \eqref{ACP} on $\prod_{j= 1}^m L_2({D}\times V)$   may now be written in matrix form 
\[
\frac{\partial }{\partial t}u_t = \boldsymbol{A}u_t, \qquad t\geq 0. 
\]
where $\boldsymbol{A} = \boldsymbol{T} + \boldsymbol{K}$ and 
\[
\boldsymbol{T} = \begin{bmatrix}
\T & \textbf{0}\\
	\textbf{0} & -\Lambda
\end{bmatrix}\quad
\text{ 
and
}
\quad
\boldsymbol{K} = \begin{bmatrix}
	\K^\circ & \M\\
	\K_\circ & \textbf{0}
\end{bmatrix}.
\]
The matrix $\boldsymbol{T}$ is an operator on $\prod_{i=1}^m L_2(D\times V))$ with domain
\[
{\rm Dom}(\boldsymbol{T}) = \prod_{i = 1}^\ell {\rm Dom}(\bT_i) \times\prod_{i = \ell+1}^m L_2(D \times V)
\]
which generates the strongly continuous semigroup 
$({\U}^{\boldsymbol{T}}_t, t\geq 0)$ given by 
\begin{equation}
{\U}^{\boldsymbol{T}}_t[g] =
\left\{
\begin{array}{ll}
{\rm e}^{-\int_0^t \sigma^i(r+\upsilon s, \upsilon)\D s}\U_t[g]& 1\leq i\leq \ell \\
 {\rm e}^{-\lambda_i t}&\ell + 1\leq i\leq m,
\end{array}
\right. 
\label{TUresolvent}
\end{equation}
for $g\in \prod_{i=1}^m L_2(D\times V))$.

\smallskip

%
%
%

\smallskip

In order to prove Theorem \ref{leadingeval}, we consider a different operator that is related to $A$ as follows. Consider the eigenvalue problem
\begin{equation}
\boldsymbol{A}\varphi = \lambda \varphi, \quad \lambda > -\lambda_{\ell +1},
\label{Apsi}
\end{equation}
for  $\varphi \in \prod_{i  =1}^m L_2(D\times V)$. Write 
\[
\varphi^\circ(\cdot) =(\varphi(1,\cdot),\cdots, \varphi(\ell, \cdot) )\text{ and } \varphi_\circ(\cdot) =(\varphi(\ell+1,\cdot),\cdots, \varphi(m, \cdot) )
\]
 so that $\varphi$ is the concatenation $(\varphi^\circ, \varphi_\circ)$. Separating this into prompt and delayed initial emissions, it can be written as 
\begin{align*}
&\T \varphi^\circ + \K^\circ \varphi^\circ + \M\varphi_\circ = \lambda \varphi^\circ\\
&\lambda{\texttt{I}}_{m-\ell}\varphi_\circ= -\Lambda\varphi_\circ + \K_\circ\varphi^\circ,
\end{align*}
where $\texttt{I}_{m-\ell}$ is the $(m-\ell)\times(m-\ell)$ identity matrix.
Substituting the second equation into the first, we get 
\begin{equation}
\varphi_\circ = 
(\lambda\texttt{I}_{m-\ell}+\Lambda)^{-1}\K_\circ\varphi^\circ
\label{2Psis}
\end{equation}
and
\begin{equation}
(\lambda \texttt{I}_\ell - \T)^{-1}\K^\circ(\lambda)\varphi^\circ =\varphi^\circ\text{ where }\K^\circ(\lambda) = \K^\circ+ \M(\lambda\texttt{I}_{m-\ell}+\Lambda)^{-1}\K_\circ.
\label{criteriaPSI}
\end{equation}
Our strategy is to show that  there exists a $\lambda_c$ such that $(\lambda \texttt{I}_\ell - \T)^{-1}\K^\circ(\lambda)$ has a leading eigenvalue $1$, and that  this is equivalent to $\lambda_c$ being an eigenvalue of $\boldsymbol{A}$. The tool we shall use to do this is the Krein-Rutman Theorem, which we recall here for convenience in a format that is appropriate for our use; c.f. \cite[p. 286]{DL6}.
\begin{theorem}[Krein-Rutman Theorem]\label{KR}
Let $X$ be a Banach space and suppose it contains a convex cone $\mathcal{C}$ such that $\mathcal{C} - \mathcal{C}: = \{h = f-g: f, g\in \mathcal{C}\}$ is dense in $X$.
 Suppose $\mathcal{L}$ is a positive compact linear operator on $X$ such that ${r}(\mathcal{L}) \coloneqq \sup\{|\lambda| : \lambda \in \Sigma(\mathcal{L})\} > 0$, where $\Sigma(\mathcal{L})$ is the spectrum of the operator $\mathcal{L}$. Then ${r}(\mathcal{L})$ is an eigenvalue of $\mathcal{L}$ with a corresponding positive eigenfunction.
\end{theorem}

Our proof of Theorem \ref{leadingeval} requires the following intermediary result below. Before stating it, the reader is reminded that the eigenvalues $\lambda_{\ell+1},\cdots, \lambda_m$ are arranged so that $\lambda_{\ell+1}$ is the smallest. Thus, the condition $\lambda>-\lambda_{\ell+1}$ ensures that $\K^\circ(\lambda)$ is well defined. In particular, $(\lambda\texttt{I}_{m-\ell}+\Lambda)$ is invertible. We will use the obvious meaning for ${\texttt{I}}_\ell$.

\begin{prop}\label{intereval}
Under the assumptions of Theorem~\ref{leadingeval}, for each $\lambda > -\lambda_{\ell +1}$, ${r}\big((\lambda \emph{\texttt{I}}_\ell - \emph{T})^{-1}\emph{\K}^\circ(\lambda)\big)$ is the leading eigenvalue of $(\lambda \emph{\texttt{I}}_\ell - \emph{\T})^{-1}\emph{\K}^\circ(\lambda)$ with a corresponding positive eigenfunction $\varphi^\circ_{\lambda}$.
\end{prop}

\begin{proof} 
In relation to the Krein-Rutman theorem stated above, our Banach space is $X = \prod_{i = 1}^m L_2(D \times V)$ and the corresponding cone is $\mathcal{C} = \prod_{i = 1}^m L^+_2(D \times V)$. It is clear that this cone is convex, and since every $L_2$ function can be written as the difference of its positive and negative parts, $\mathcal{C}$ satisfies the assumptions of the theorem.
We now break the rest of the proof into a number of steps which are stated with a proof immediately afterwards. 
\smallskip

{\it Step 1.} First we claim that $(\lambda \texttt{I}_\ell - \T)^{-1}\K^\circ(\lambda)$ is a compact operator.
\smallskip

Fix $1\leq i,j\leq m$. By Fubini's Theorem we have that $r\mapsto \K_{i,j}f(r,\upsilon)$ is measurable for $g\in L_2(D\times V)$. The operators $ \K_{i,j}$ are also integral operators and therefore  are continuous on $L_2(V)$ and compact. The assumed piecewise continuity of the cross sections $\sigma_{\texttt s}^i\pi^i_{\texttt s}$ and $\sigma_{\texttt f}^i\pi^{i,j}_{\texttt f}$ and the boundedness of the domain $V$ is sufficient to ensure that $r\mapsto \K_{i,j}\cdot (r,\cdot)$ is continuous under the operator norm on $L_2(V)$ and hence $\{\K_{i,j}\cdot (r,\cdot): r\in D\}$ forms a relatively compact set in the space of linear operators on $L_2(V)$. With these properties, the mapping $r\mapsto \K_{i,j}\cdot (r,\cdot)$, for $r\in D$, is said to be {\it regular}. One similarly (but more easily) shows that $r\mapsto \M_{i,j}\cdot(r,\cdot)$ is regular for $r\in D$ as operators on $L_2(V)$.
By linearity, this implies that, for $1\leq i,j\leq\ell$, the mapping  $r\mapsto K^\circ(\lambda)_{i,j}$ is also regular. Hence, by~\cite[Theorem 4.1]{M-K}, $(\lambda \texttt{I}_\ell - \T)^{-1}\K^\circ(\lambda)$ is a compact operator. 

\begin{remark}\rm It is precisely at the application of \cite[Theorem 4.1]{M-K} that we need the convexity of the domain $D$, as this is required within the aforesaid result.
\end{remark}
\smallskip

{\it Step 2.} Next we show that $(\lambda \texttt{I}_\ell- \T)^{-1}\K^\circ(\lambda)$ is a positive irreducible operator.

\smallskip

Positivity is a straightforward consequence of the assumptions on the operators $\K_{i,j}$ and the form of the semigroup defined in ~\eqref{TUresolvent}.
For irreducibility, it is enough to show that there exists an integer $n \ge 1$ such that $[(\lambda \texttt{I}_\ell- \T)^{-1}\K^\circ(\lambda)]^nf > 0$ for each $f\in\prod_{i=1}^\ell L^+_2(D\times V)$. 
To this end, note that the entries of $\K^\circ(\lambda)(\lambda \texttt{I}_\ell- \T)^{-1}\K^\circ(\lambda)$ satisfy
\begin{align*}
[\K^\circ(\lambda)(\lambda \texttt{I}_\ell- \T)^{-1}\K^\circ(\lambda)]_{i,j} &\ge [\K^\circ(\lambda \texttt{I}_\ell - \T)^{-1}\K^\circ]_{i,j} = \sum_{k=1}^\ell  \K_{i,k}(\lambda  \texttt{I}_\ell- \T)^{-1}_{k,k}\K_{k,j},
\end{align*} 
and that $\K_{i,k}(\lambda \texttt{I}_\ell - \T)^{-1}_{k,k}\K_{k,j}$ is an integral operator $L_2(D\times V) \to L_2(D \times V)$, $1\leq i,j\leq \ell$, whose kernel is greater than or equal to
\begin{equation}
\int_0^{\infty}{\rm e}^{-\lambda +\underline{\sigma}^k(\frac{r-r'}{t}))t}\texttt{k}_{i,k}\left(r',\frac{r-r'}{t}, v'' \right)\texttt{k}_{k,j}\left(r,v, \frac{r-r'}{t} \right)\frac{\D t}{t^n},
\label{lb}
\end{equation}
where $\underline\sigma^k(v) = \inf_{r\in D}\{\sigma^k(r,\upsilon)\}$. 
Note, in order to produce this estimate, the reader will note that $(\lambda \texttt{I}_\ell - \T)^{-1}_{k,k}$ is the resolvent of $(\texttt{U}^{\boldsymbol{T}}_t, t\geq 0)$ in \eqref{TUresolvent}.
If we choose the index $k$ as in the assumptions~\eqref{irred1} and~\eqref{irred2} then the lower bound \eqref{lb} ensures that  $[\K^\circ(\lambda)(\lambda \texttt{I}_\ell - \T)^{-1}\K^\circ(\lambda)]_{i,j}$ is positivity improving. It follows that $[(\lambda \texttt{I}_\ell- \T)^{-1}\K^\circ(\lambda)]^2$ is also positivity improving and therefore $(\lambda \texttt{I}_\ell- \T)^{-1}\K^\circ(\lambda)$ is irreducible.
  
\smallskip

{\it Step 3}. We claim that there exists a non-negative eigenfunction  $ 0\neq\varphi_\lambda\in \prod_{i = 1}^\ell L_2(D\times V) $ for the operator $(\lambda \texttt{I}_\ell- \T)^{-1}\K^\circ(\lambda)$ with 
eigenvalue that agrees with ${r}\big((\lambda \texttt{I}_\ell - \T)^{-1}\K^\circ(\lambda)\big)$.
\smallskip

We use de Pagter's Theorem, cf. \cite[Theorem 5.7]{M-K}, which says that the spectral radius of an irreducible operator is strictly positive; that is to say ${r}\big((\lambda \texttt{I}_\ell - \T)^{-1}\K^\circ(\lambda)\big)>0$. In turn the Krein-Rutman Theorem \ref{KR} states that ${r}\big((\lambda \texttt{I}_\ell - \T)^{-1}\K^\circ(\lambda)\big)$ is thus an eigenvalue for the operator $(\lambda \texttt{I}_\ell- \T)^{-1}\K^\circ(\lambda)$ with a corresponding non-negative eigenfunction $\varphi^\circ_{\lambda}$.
\end{proof}

\begin{proof}[Proof of Theorem \ref{leadingeval}](i)
In looking for a non-negative eigenfunction of $\bA$ with real eigenvalue,  our earlier discussion tells us we must equivalently look for a solution to \eqref{Apsi} and hence \eqref{criteriaPSI}. This is equivalent to finding a real value  $\lambda_c$ such that 
$r\big((\lambda_c \texttt{I}_\ell - \T)^{-1}\K^\circ(\lambda_c)\big) = 1$. We again achieve this goal in steps.

\bigskip

{\it Step 1}. We want to show that 
 \begin{equation}
\lim_{\lambda \downarrow -\lambda_{\ell +1}}{r}\big((\lambda \texttt{I}_\ell- \T)^{-1}\K^\circ(\lambda)\big) = \infty.
\label{tooo}
\end{equation} \smallskip
 
Recall that  $(\lambda \texttt{I}_\ell- \T)^{-1}\K^\circ(\lambda)$ is compact and irreducible so by~\cite[Theorem 5.13]{M-K} we have the comparison of the spectral radii, 
\begin{equation}
{r}\big((\lambda \texttt{I}_\ell- \T)^{-1}\K^\circ(\lambda)\big) \ge {r}\big((\lambda \texttt{I}_\ell- \T)^{-1}\Delta[{\K}^\circ(\lambda)]\big),
\label{ineq1}
\end{equation}
where $\Delta[{\K}^\circ(\lambda)]$ is the matrix whose entries are given by $\Delta[{\K}^\circ(\lambda)] =\diagonal( \K^\circ(\lambda)_{1,1}, \cdots, \K^\circ(\lambda)_{\ell,\ell} )$.
\smallskip

Suppose $\Delta$ is an $\ell\times \ell$  whose diagonal entries are given by operators $\Delta_i$ on $L_2(D\times V)$, for $i = 1,\cdots, \ell$. If $\mu \in \sigma(\Delta_1)$, the spectrum of $\Delta_1$, then $(\mu \texttt{I}_\ell - \Delta)_{1,1}$ is not invertible, and so $\mu \texttt{I}_\ell - \Delta$ is also not invertible. Hence $\mu \in \sigma(\Delta)$, the spectrum of $\Delta$, and so $\sigma(\Delta_1) \subset \sigma(\Delta)$. Applying this  argument to the diagonal matrix $(\lambda \texttt{I}_\ell- \T)^{-1}\Delta{\K}^\circ(\lambda)$, we have that 
\begin{equation}
\sigma([(\lambda \texttt{I}_\ell- \T)^{-1}\Delta{\K}^\circ(\lambda)]_{1,1}) \subset \sigma((\lambda \texttt{I}_\ell- \T)^{-1}\Delta{\K}^\circ(\lambda))
\label{inclusion}
\end{equation}
and so
\begin{equation}
{r}\big((\lambda \texttt{I}_\ell- \T)^{-1}\Delta{\K}^\circ(\lambda)\big) \ge {r}\big([(\lambda \texttt{I}_\ell- \T)^{-1}\Delta {\K}^\circ(\lambda)]_{1,1}\big)
\geq {r}\big((\lambda - \bT_1-\sigma^1)^{-1}\Delta [{\K}^\circ(\lambda)]_{1,1}\big).
\label{ineq2}
\end{equation}
where, in the final inequality, we have used \eqref{inclusion}.

\smallskip

Next recall that $(\lambda \texttt{I}_\ell - \T)^{-1}\K^\circ(\lambda)\varphi^\circ =\varphi^\circ\text{ where }\K^\circ(\lambda) = \K^\circ+ \M(\lambda\texttt{I}_{m-\ell}+\Lambda)^{-1}\K_\circ$. Similar reasoning to the proofs of previous steps shows us that 
$
(\lambda - \bT_1-\sigma^1)^{-1}\Delta [{\K}^\circ(\lambda)]_{1,1}$ and $(\lambda - \bT_1-\sigma^1)^{-1}\sigma^1_{\texttt f}m^{\ell+1}(\K_\circ)_{1,\ell+1}    $ are both compact and irreducible operators, so that
\begin{equation}
{r}\big((\lambda - \bT_1-\sigma^1)^{-1}\Delta [{\K}^\circ(\lambda)]_{1,1}\big) \ge 
\frac{{r}\big((\lambda - \bT_1-\sigma^1)^{-1}\sigma^1_{\texttt f}m^{\ell+1}(\K_\circ)_{1,\ell+1}    \big)}{\lambda + \lambda_{\ell +1}} > 0,
\label{ineq3}
\end{equation}
where the first inequality follows from~\cite[Theorem 5.13]{M-K} and the second follows from~\cite[Theorem 5.7]{M-K}.
Combining~\eqref{ineq1},~\eqref{ineq2} and~\eqref{ineq3}, we have
\[
{r}\big((\lambda \texttt{I}_\ell- \T)^{-1}\K^\circ(\lambda)\big) \ge \frac{{r}\big((\lambda - \bT_1-\sigma^1)^{-1}\sigma^1_{\texttt f}m^{\ell+1}(\K_\circ)_{1,\ell+1}  \big)}{\lambda + \lambda_{\ell +1}}  > 0,
\]
with the latter term tending to $\infty$ as $\lambda \to -\lambda_{\ell +1}$.

\smallskip

{\it Step 2.} Next we need to show that 
\[
\lim_{\lambda \to \infty}r\big((\lambda \texttt{I}_\ell- \T)^{-1}\K^\circ(\lambda)\big) < 1
\]  
\smallskip

The spectral radius ${r}\big((\lambda \texttt{I}_\ell- \T)^{-1}\K^\circ(\lambda)\big)$ as is $\K^\circ(\lambda)$. Using the standard operator norm $\norm{\cdot}_2$ on $\prod_{i=1}^\ell L_2(D\times V)$,
 \[
\norm{\K^\circ(\lambda) g}_2=\norm{\M\left(
\diagonal\big((\lambda +\lambda_{\ell+1})^{-1},\cdots, (\lambda +\lambda_{m})^{-1}\big)
\right)
\K_\circ g}_2
 \]  
 and, hence, by inspection, $\K^\circ(\lambda)$ is decreasing with $\lambda$ and
 tends to $\K^\circ$ as $\lambda \to \infty$. Note, moreover, that for all 
$g\in\prod_{i=1}^\ell L_2(D\times V)$, 
\[
(\lambda \texttt{I}_\ell- \T)^{-1}\K^\circ(\lambda) g= \int_0^\infty {\rm e}^{-\lambda t}\langle f, \texttt{U}^{\boldsymbol{T}}_t[\K^\circ(\lambda) g]\rangle \d t,
\]
showing similarly that $(\lambda \texttt{I}_\ell- \T)^{-1}\K^\circ(\lambda)$ is decreasing in $\lambda$.
 Due to~\cite[Lemma 8.1]{M-K} (note that it is not difficult to see from the proof of that lemma that that the order of the operators there  can be reversed), we have 
\[
\lim_{\lambda \to \infty}r\big((\lambda \texttt{I}_\ell- \T)^{-1}\K^\circ(\lambda)\big) < 1.
\] 

{\it Step 3.} In this penultimate step, we show that 
we have found a non-negative function of $\bA$, with eigenvalue $\lambda_c$.
\smallskip

We  have the existence of a  $\lambda_c > -\lambda_{\ell +1}$ such that ${r}((\lambda \texttt{I}_\ell- \T)^{-1}\K^\circ(\lambda))= 1$. 
That is to say, thanks to Proposition \ref{intereval}, we have found $\varphi^\circ = \varphi^\circ_{\lambda_c}$ which solves \eqref{criteriaPSI}, which in turn, thanks to \eqref{2Psis} gives us that $\varphi_\circ = 
(\lambda\texttt{I}_{m-\ell}+\Lambda)^{-1}\K_\circ\varphi^\circ_{\lambda_c}$ so that with the concatenation 
\[
\varphi = (\varphi^\circ_{\lambda_c}, (\lambda_c\texttt{I}_{m-\ell}+\Lambda)^{-1}\K_\circ\varphi^\circ_{\lambda_c})\geq 0
\]
we have the eigensolution
\[
\boldsymbol{A}\varphi = \lambda_c \varphi.
\]
which is equivalent to $\bA\varphi = \lambda_c\varphi$.

\smallskip

{\it Step 4.} For the final step we need to show that $\lambda_c$ is the leading real eigenvalue of $\bA$, i.e.
\[
\lambda_c = s(\boldsymbol{A}) \coloneqq \sup\{{\rm Re}(\lambda) : \lambda \in \sigma(\boldsymbol{A})\},
\]
where $\sigma(\boldsymbol{A})$ is the spectrum of the operator $\boldsymbol{A}$ or equivalently of $\bA$.  Moreover we need to show that it is simple and isolated.
\smallskip

We first note that since we have shown that $\lambda_c \in \sigma(\boldsymbol{A})$, in particular that the spectrum is non-empty, it follows from~\cite[Theorem 5.2]{M-K} that $s(\boldsymbol{A}) \in \sigma(\boldsymbol{A})$. Now suppose that $\lambda_c \neq s(\boldsymbol{A})$ so that, in particular, $\lambda_c < s(\boldsymbol{A})$. Then, thanks again to~\cite[Lemma 8.1]{M-K},  $r\big((s(\boldsymbol{A}) \texttt{I}_{\ell}- \T)^{-1}\K^\circ(s(\boldsymbol{A}))\big) < 1$ and so $1$ is not an eigenvalue of $(s(\boldsymbol{A}) \texttt{I}_{\ell}- \T)^{-1}\K(s(\boldsymbol{A}))$. Said another way, this means that $s(\boldsymbol{A})$ is not an eigenvalue of $\boldsymbol{A}$ (and hence of $\bA$), leading to a contradiction. Algebraic and geometric simplicity of $\lambda_c$ follows from~\cite[Remark 12]{DL6} and~\cite[Theorem 7(iii)]{DL6}, respectively.  
\end{proof}

Before turning to the proof of Theorem \ref{leadingeval} (ii), we must state another intermediary result which is translated from a general setting of Banach operators to our current situation; cf. \cite[Theorem 4.1]{M-K} and~\cite[p. 359, Theorem 22]{BMK}.
\begin{prop}\label{discretespectra}
Under the assumptions of Theorem \ref{leadingeval}
\[
\sigma(\boldsymbol{A}) \cap \{{\rm Re}(\lambda) : \lambda >s(\boldsymbol{T}) \}
\]
consists of isolated eigenvalues with finite multiplicities, where $s(\boldsymbol{T})  \coloneqq \sup\{{\rm Re}(\lambda) : \lambda \in \sigma(\boldsymbol{T})\}$. 
\end{prop}

Note the Theorem from which the above proposition is derived in  \cite[p. 359, Theorem 22]{BMK} requires as a sufficient condition that $(\lambda {\boldsymbol I} - \boldsymbol{T})^{-1}\boldsymbol{K}$ is compact, where $\boldsymbol{I}$ is an $m\times m$ identity matrix. This fact easily follows from the conclusion in Step 1 of the proof of Proposition \ref{intereval}. 
\smallskip

Finally we can complete the proof of Theorem \ref{leadingeval}

\begin{proof}[Proof of Theorem \ref{leadingeval}](ii)
It is also easy from the structure of $\boldsymbol{T}$ that $-\lambda_{\ell+1},\cdots, -\lambda_{m}$, belong to its spectrum. Moreover, for all $i = 1,\cdots, \ell$, $s(\bT_i - \sigma^i) = -\infty$. Since $-\lambda_{\ell + 1}$ is the largest of these eigenvalues, and $\lambda_c>-\lambda_{\ell + 1}$ (from part (i) of Theorem~\ref{leadingeval}), Proposition \ref{discretespectra} tells us that  $\sigma(\boldsymbol{A}) \cap \{\lambda : \rm{Re}(\lambda)>-\lambda_{\ell +1}\}$ contains at least one  isolated eigenvalue with finite (algebraic) multiplicity (i.e. the lead eigenvalue $\lambda_c$). \smallskip

Suppose we enumerate the eigenvalues in $\sigma(\boldsymbol{A}) \cap \{\lambda : \rm{Re}(\lambda)>-\lambda_{\ell +1}\}$ in decreasing order by the set $\{\lambda^{(1)}, \cdots, \lambda^{(n)}\}$ (noting from earlier that we have at least $\lambda^{(1)}= \lambda_c$ and $\lambda^{(n)}> -\lambda_{\ell + 1}$). Then, from \cite[p. 265]{DL6}, for $g \in {\rm Dom}(\bA)$, we have
\[
\texttt{V}_t[g] = \sum_{k=1}^n {\rm e}^{\lambda^{(k)} t }\left(\sum_{m = 0}^{{\rm order}(\lambda^{(k)} ) - 1}t^m\Pi_k^m [g]\right) + O({\rm e}^{-\lambda_{\ell +1} t}),
\]
as $t\to \infty$, where $\Pi_k$ are projectors in ${\rm Dom}(\bA)$. \smallskip

We are really only interested in the projection onto the  eigenfunction that we know exists in the real part of the spectrum. The projector $\Pi_1$ can be written in the form
\[
\Pi_1[g] = \langle g, \tilde{\varphi}\rangle\varphi, \quad g\in \prod_{i= 1}^m L_2({D}\times V),
\]
where $\tilde\varphi$ is the left-eigenfunction with eigenvalue $\lambda_c$, which is guaranteed to exist by examining the preceding arguments for $\bA$ and re-applying them for $\fA := \fT+\fS+\fF$, the adjoint operator of $\bA$.
Hence, we have the following leading order expansion,
\begin{align*}
V_t[f] &= {\rm e}^{\lambda_ct}(f, \tilde{\varphi})\varphi   +O({\rm e}^{[\lambda^{(2)}\vee(-\lambda_{\ell+1} )]t}).
\end{align*}
Note that since, according to Proposition \ref{discretespectra},  $\lambda_c$ is isolated, there exists a $\varepsilon>0$ such that $\lambda^{(2)}\vee(-\lambda_{\ell+1} )<\lambda_c- \varepsilon$, where we understand $\lambda^{(2)} = -\infty$ if $n = 1$.
The statement of part (ii) of Theorem \ref{leadingeval} now follows. 
\end{proof}

\section*{Acknowledgements}  We are indebted to Paul Smith and Geoff Dobson from the ANSWERS modelling group at {\it Wood} for the extensive discussions as well as hosting at their offices in Dorchester. We would also like to thanks Minmin Wang, Ivan Graham, Matt Parkinson and Denis Villemonais  for useful discussions. Finally we would like to thank an enthusiastic referee for their comments and support of this article which is part review, part new results.

\section*{Glossary} For convenience, at the request of the referee, we include a short glossary of the shorthand terminology. 

\begin{table}[h]
\begin{tabular}{l l l}
&&\\
\hline
Abbreviation & Description & Introduced \\
\hline
&&\\
NTE  & Neutron Transport Equation &  \eqref{NTE}, \eqref{bNTE} \\
MNTE  & Multi-species Neutron Transport Equation &  \eqref{promptNTE}, \eqref{delayNTE} \\
MNBP& Multi-species Neutron Branching Process& \S\ref{MNBPsect}\\
MNRW& Multi-species Neutron Random Walk &\S\ref{RW}\\
  & & \\
\hline

\end{tabular}
\label{table-notation}
\end{table}

\bibliography{references}{}
\bibliographystyle{plain}

\end{document}